\documentclass[a4paper,10pt]{article}
\usepackage[utf8]{inputenc}
\usepackage{amsmath,amssymb,amsthm}
\usepackage{mathtools} 					
\usepackage{microtype} 					
\usepackage{bbm}					
\usepackage[shortcuts]{extdash} 			
\usepackage{color}
\usepackage{enumerate}
\usepackage[colorlinks=true,linkcolor=black, citecolor=black]{hyperref}  
\usepackage{mparhack}					
\usepackage{accents}

\usepackage{tikz}
\usetikzlibrary{arrows,matrix}
\tikzset{
>=latex
}

\renewcommand{\leq}{\leqslant}
\renewcommand{\geq}{\geqslant} 

\newcommand{\RR}{\mathbb{R}}
\newcommand{\QQ}{\mathbb{Q}}

\newcommand{\NN}{\mathbb{N}}

\renewcommand{\SS}{\mathbb{S}}

\newcommand{\infun}{\mathbbm{1}}			

	\newcommand{\CF}{\mathcal{F}}	
\newcommand{\CG}{\mathcal{G}}	\newcommand{\CH}{\mathcal{H}}

	\newcommand{\CN}{\mathcal{N}}	
\newcommand{\CO}{\mathcal{O}}	\newcommand{\CP}{\mathcal{P}}

\newcommand{\CU}{\mathcal{U}}	\newcommand{\CV}{\mathcal{V}}	
	\newcommand{\CX}{\mathcal{X}}	
\newcommand{\CY}{\mathcal{Y}}			

\newcommand{\dist}[1]{{d^{#1}}}
\newcommand{\wdist}[1]{{\delta^{#1}_\Gamma}} 
\newcommand{\cone}{\CO} 

\newcommand{\bigmid}{\mathrel{\big|}}			
\newcommand{\Bigmid}{\mathrel{\Big|}}			

\newcommand{\note}[1]{\iffalse \marginpar{\footnotesize #1}\fi}  
								      
\DeclareMathOperator{\diam}{diam}

\DeclareMathOperator{\vertex}{V}

\DeclareMathOperator{\injrad}{inj}

\DeclareMathOperator{\SO}{SO}
\DeclareMathOperator{\SU}{SU}
\DeclareMathOperator{\U}{U}
\DeclareMathOperator{\Gl}{Gl}
\DeclareMathOperator{\Sl}{Sl}

\DeclareMathAlphabet{\mathbit}{OT1}{cmr}{bx}{it}  	

\DeclareSymbolFont{bbold}{U}{bbold}{m}{n} 	
\DeclareSymbolFontAlphabet{\mathbbold}{bbold}  

\DeclarePairedDelimiter\abs{\lvert}{\rvert}		
\DeclarePairedDelimiter\norm{\lVert}{\rVert}		
\DeclarePairedDelimiter\angles{\langle}{\rangle}	
\makeatletter
\let\oldabs\abs	
\def\abs{\@ifstar{\oldabs}{\oldabs*}}			
 
\let\oldnorm\norm
\def\norm{\@ifstar{\oldnorm}{\oldnorm*}}
\makeatother

\DeclarePairedDelimiter\paren{(}{)}			
\DeclarePairedDelimiter\brackets{[}{]}			
\DeclarePairedDelimiter\braces{\{}{\}}			
\newcommand{\bigpar}[1]{\paren[\big]{#1}}	\newcommand{\bigparen}[1]{\paren[\big]{#1}}
\newcommand{\Bigpar}[1]{\paren[\Big]{#1}}	\newcommand{\Bigparen}[1]{\paren[\Big]{#1}}
\newcommand{\bigbrack}[1]{\brackets[\big]{#1}}

	\newcommand{\bigbraces}[1]{\braces[\big]{#1}}

\theoremstyle{plain}
\newtheorem{thm}{Theorem}[section]
\newtheorem{prop}[thm]{Proposition}
\newtheorem{lem}[thm]{Lemma}
\newtheorem{cor}[thm]{Corollary}
\newtheorem{qu}[thm]{Question}
\newtheorem{alphthm}{Theorem}

\newtheorem{conj}[thm]{Conjecture}

\theoremstyle{definition}
\newtheorem{de}[thm]{Definition}

\theoremstyle{remark}
\newtheorem{rmk}[thm]{Remark}

\title{Measure expanding actions, expanders and warped cones}
\author{Federico Vigolo}
\date{}

\begin{document}

\maketitle

\begin{abstract}
 We define a way of approximating actions on measure spaces using finite graphs; we then 
show that in quite general settings these graphs form a family of expanders if and only 
if the action is expanding in measure. This provides a somewhat unified approach to 
construct expanders.
We also show that the graphs we obtain are uniformly 
quasi\=/isometric to the level sets of warped cones. This way we can also prove 
non\=/embeddability results for the latter and restate an old conjecture of 
Gamburd-Jakobson-Sarnak.
\end{abstract}

\section{Introduction}

In this paper we explore the geometry of finite 
graphs obtained by approximating measurable actions of groups on metric spaces. As a 
consequence we construct new families of expanders by approximating measure 
preserving actions with spectral gaps. 

The geometric nature of our construction can be used to shed some light on the geometry 
underlying some problems related to expanders. 
In particular, we show that families of expanders are closely related with warped cones 
and we produce numerous examples of warped cones which coarsely contain a family of 
expanders and hence do not coarsely embed into any $L^p$ space.

\

{\bf Families of expanders.}
Expander graphs are graphs at the same time sparse and highly connected. They 
were first defined in the $70$s and they immediately found important applications to 
applied mathematics and computer science. Soon enough, their interest was recognised also 
in various fields of pure mathematics, for example in the theory of coarse embeddability 
of 
metric spaces. See \cite{HLW06} and \cite{Lub12} for a survey on the subject.

The existence of families of expanders was first proved by Pinsker \cite{Pin73} by 
probabilistic means. Indeed, he showed that random sequences of graphs with bounded 
degrees are 
expanders with high probability. Despite his results, it turned out that 
defining explicit families of expanders was a challenge. The first 
explicit examples of expander graphs were built by Margulis \cite{Mar73} using the 
machinery of Kazhdan's property (T). 

To give the precise definition, for a given constant $\varepsilon>0$ we say that a graph 
$\CG$ satisfies an $\varepsilon$-\emph{linear isoperimetric inequality} if every 
subset of vertices $W\subset\vertex(\CG)$ with $\abs{W}\leq\abs{\vertex(\CG)}/2$ 
satisfies $\abs{\partial W}\geq \varepsilon\abs{W}$. 
Here $\partial W$ denotes the exterior vertex 
boundary \emph{i.e.} the set of vertices in $\vertex(\CG)\smallsetminus W$ linked by an 
edge to vertices in $W$. 
The largest constant $\varepsilon$ so that $\CG$ satisfies an $\varepsilon$\=/linear 
isoperimetric inequality is known as the 
\emph{Cheeger constant} of $\CG$. Note that in literature different 
kinds of boundaries are often used, but this is not an issue because they are all 
coarsely 
equivalent when the considered graphs have bounded degree.

A \emph{family of expanders} is a sequence of finite graphs $\CG_n$ of 
uniformly bounded degree with $\abs{\vertex(\CG_n)} \to \infty$ such that they all 
satisfy an $\varepsilon$-linear isoperimetric inequality where the constant 
$\varepsilon>0$ is fixed. Notice that 
every finite graph satisfies an isoperimetric inequality with positive $\varepsilon$ as 
long as it is connected, thus 
the important bit here is the existence of a constant $\varepsilon>0$ independent of 
$n$. In this sense, the above is a notion of strong connectedness.

Various different constructions of expanders are known nowadays. They are 
generally built either using zig-zag products, or techniques from additive 
combinatorics, or 
infinite dimensional representation theory. The main goal of this work is to illustrate a 
general procedure to build families of expanders out of actions on measure spaces, 
partially expanding and linking the latter two techniques. Our approach is rather 
geometric in nature and it turns out to be fruitful in the study of warped cones
as well.

\

{\bf Coarse embeddings and warped cones.}
Since the late 50's, many mathematicians have been interested in studying different 
notions 
of embeddings of metric spaces into Hilbert spaces. It was known that every separable 
metric space is homeomorphic to a subset of $L^2(0,1)$, but in \cite{Enf70} Enflo 
provided an example of a separable metric space which does not 
embed into a Hilbert space \emph{uniformly}.
It was later asked by Gromov \cite[p218]{Gro93} whether every separable metric 
space embeds into a Hilbert space \emph{coarsely}. 

It was proved by Yu in \cite{Yu00} 
that the coarse Baum\=/Connes (and hence the Novikov conjecture) holds true for every 
metric spaces which coarsely embed into a Hilbert space. Yu's result combined with 
Gromov's question provided a valid strategy to tackle such conjectures, but in 
\cite{DGLY02} it was built an example of a separable metric 
space that does not coarsely embed into an Hilbert space. Their example was constituted 
by a family of locally finite graphs with growing degrees and therefore it did not have 
\emph{bounded geometry} (recall that a metric space $X$ has bounded geometry if for every 
$\varepsilon,R>0$ there exists an $N\in\NN$ such that any $\varepsilon$\=/separated 
subset of an $R$\=/ball of $X$ has at most $N$ elements).

It was pointed out by 
Gromov that the reason why those graphs could not coarsely embed was related with their 
expansion properties and he also noted that families of expanders do not coarsely embed 
into Hilbert spaces. This idea led Higson to produce a counterexample to the coarse 
Baum\=/Connes conjecture \cite{Hig99}. In \cite{Gro03}, Gromov used 
probabilistic methods to build finitely generated groups that contain a coarsely embedded 
family of expanders and therefore do not coarsely embed into Hilbert spaces. These groups 
where used in \cite{HLS02} to produce counterexamples to the Baum\=/Connes conjecture.
It is in this context that John Roe first defined the warped cones. 

\

Following 
\cite{Roe05}, let $(M,\varrho)$ be a compact Riemannian manifold and let a finitely 
generated group 
$\Gamma=\angles{S}$ act by diffeomorphisms on $M$. The \emph{warped cone} associated with 
this action is the metric space $(\cone_\Gamma(M),\wdist{})$ obtained from the infinite 
Riemannian cone 
$\big(M\times [1,\infty),t^2\varrho+dt^2\big)$ by warping the metric, imposing the 
condition that the 
distance between any two points of the form $(x,t)$ and $(s\cdot x,t)$ is at most $1$. 
Warped cones are spaces with bounded geometry and with a very 
large group of translations, but their most interesting features regard their coarse 
geometry (which does not depend on the choice of the generating set). Indeed, Roe showed 
that this construction is flexible enough to produce examples of spaces with or without 
property A of Yu: warped cones of amenable actions have property A and, for a particular 
kind of actions, the converse is also true. Moreover, warped cones also provide further 
examples of spaces that do not coarsely embed into Hilbert spaces.
 
The counter examples to the coarse Baum\=/Connes conjecture built in \cite{Hig99,HLS02} 
are based on the existence, for a coarse disjoint union of expanders $X$, of 
non\=/compact ghost projections for the Roe algebra $C^* (X)$. In \cite{DrNo15} it is 
shown that also warped cones can have such ghost projections. This means that they are 
good candidates for a new class of counterexamples to the coarse Baum\=/Connes 
conjecture. To the present day, the only known counterexamples to Baum\=/Connes are 
constructed using expanders, it is therefore of interest to understand how families of 
expanders relate to warped cones.

If $\Gamma$ is a finitely generated subgroup of a compact Lie group $G$ one can construct 
a warped cone $\cone_\Gamma(G)$ by considering the action by left multiplication.
In \cite{Roe05} it is proved that, in this setting, if the warped cone 
$\cone_\Gamma(G)$ coarsely embeds in a Hilbert space then the group $\Gamma$ must satisfy 
an analytical weak version of amenability: the Haagerup property. It is a natural 
question to ask the following: 
\begin{qu}
\label{qu:Roe.question}
 Is it true that if $\Gamma$ has the Haagerup property and acts freely by isometries on a manifold $M$ then the warped cone 
$\cone_\Gamma(M)$ coarsely embeds into a Hilbert space? 
\end{qu}

As a consequence 
of our result we can give a negative answer to this question. Indeed, 
there exist many warped cones for $G$ over a non-abelian free group $\Gamma$ which 
cannot be 
embedded into a Hilbert space for the strongest of reasons: they contain a family of 
expanders. 

\

{\bf Expanders \emph{via} approximating graphs and other results.}
Given a finitely presented group $\Gamma$ with a finite 
generating set $S$ and a measurable action on a probability space 
$\Gamma\curvearrowright(X,\nu)$, one can try to `approximate' this action \emph{via} a 
finite graph. Specifically, if we choose a finite partition $\CP$ of $X$ into measurable 
subsets, then we can define an \emph{approximating graph} by considering the graph 
$\CG(\CP)$ whose vertices are the regions of the partition and such that two regions 
$R,R'\in\CP$ are linked by an edge if there exists 
an element $s$ in the generating set $S$ so that the image $s(R)$ intersect $R'$ 
non\=/trivially.

We say that such an action is \emph{expanding in measure} if there exists a constant 
$\alpha>0$ so that for every measurable set $A\subset X$ with $\nu(A)\leq 1/2$ the union 
of the images $s(A)$ with $s\in S$ has measure at least $(1+\alpha)\nu(A)$.
Our main tool will then be the following:

\begin{alphthm}
\label{thm:intro:approx.exp_iff_meas.exp}
Let $(X,d)$ be a locally compact metric space and $\nu$ a Radon measure thereon. Moreover,
let $\Gamma=\angles{S}$ be a finitely generated group with a continuous action 
$\Gamma\curvearrowright (X,d)$. 
Assume that $\CP_n$ is a family of 
measurable partitions of $X$ which are `regular enough' and so that the 
diameters of most of the regions tend to zero. Then the approximating graphs $\CG(\CP_n)$ 
all satisfy an isoperimetric inequality with the same Cheeger constant $\varepsilon>0$ if 
and only if the action is expanding in measure. 
\end{alphthm}

We refer the reader to Section \ref{sec:meas.exp.and.Cheeger.const} and Theorem 
\ref{thm:meas.exp_iff_unif.Cheeger.const} for precise statements and definitions.

During the proof of Theorem \ref{thm:intro:approx.exp_iff_meas.exp} it is shown 
that the Cheeger constant $\varepsilon$ and 
the expansion constant $\alpha$ depend rather explicitly on each other. Moreover, the 
same statement holds true also when the graphs and the total measure $\nu(X)$ are 
infinite. 

\begin{rmk}
 It was already known that when a graph comes naturally from an action on measure spaces 
then one can estimate its Cheeger constant by studying expanding properties of the 
action. In fact, this has been used more or less implicitly 
in many works on expanders (\emph{e.g.} \cite{Mar73,GaGa81,Sha97}). 
Despite this, to our knowledge nobody actually made explicit an equivalence such as that 
of Theorem \ref{thm:intro:approx.exp_iff_meas.exp}.
\end{rmk}

It appears to us, that our method is well suited to constructing numerous families 
of expanders. 
We also believe that the intrinsic geometric 
nature of the approximating graphs can be successfully used to provide some intuition in 
settings whose nature could otherwise be rather obscure.
An excellent example in this sense is given by the work of Bourgain and Yehudayoff: in 
\cite{BoYe13} they managed to build  the first known families of \emph{monotone 
expanders} using techniques quite similar to ours.

As a further example of the convenience of our geometric intuition, we can use the 
geometric control we have on our construction to find a close connection between 
approximating graphs and warped cones. 
Indeed, let $\Gamma$ act by diffeomorphisms on a compact Riemannian manifold $M$ and let 
$\cone_\Gamma(M)$ be the resulting warped cone.
The coarse structure of the level sets of the warped cone
$M\times\{t\}\subset\cone_\Gamma(M)$ with respect to the restriction of the warped metric 
$\wdist{}$ is actually equal to the coarse structure of some graphs approximating the 
action $\Gamma\curvearrowright M$. As a consequence, we prove the following:

\begin{alphthm}
 \label{thm:intro:warped.cones.exp_iff_meas.exp}
 Let $(\cone_\Gamma(M),\wdist{})$ be the warped cone built from an action of a finitely 
generated group $\Gamma$ on a compact Riemannian manifold $M$. 
Then for every sequence $t_n\to\infty$ the level sets 
$(M\times\{t_n\},\wdist{})$ are uniformly quasi\=/isometric to a sequence of expander 
graphs if and only if the action $\Gamma\curvearrowright M$ is expanding in measure.
\end{alphthm}

Since it is known that expander graphs do not coarsely embed into any $L^p$ space (see 
\cite{Mat97}), we get as a corollary that the if the $\Gamma$\=/action is expanding in 
measure the warped cone 
$\cone_\Gamma(M)$ does not embed into any $L^p$ space either. This also provides an 
alternative proof of the main result of \cite{NoSa15}.

Theorem \ref{thm:intro:warped.cones.exp_iff_meas.exp} gives a strongly negative answer to 
Question \ref{qu:Roe.question}. Indeed, 
there are numerous examples of actions of non\=/abelian free groups which are expanding 
in measure (see \cite{BoGa07}), thus the warped cone does not embed in 
any Hilbert space. Nevertheless, free groups do have the Haagerup property. 

\

Concrete examples may be produced using the fact that a measure 
preserving action on a probability space $\Gamma\curvearrowright (X,\nu)$ is 
expanding in measure if and only if the induced unitary representation $\pi\colon 
\Gamma\curvearrowright L^2(X)$ has a spectral gap. In particular, given a measure 
preserving action of a locally compact group on probability space $G\curvearrowright 
(X,\nu)$, it follows that its restriction to a finitely generated subgroup 
$\Gamma=\angles{S}<G$ is expanding in measure if and only if $S$ is a Kazhdan set for the 
representation $G\curvearrowright L^2(X)$ (see Definition \ref{de:Kazhdan_pairs}).

The spectral properties of unitary representations are a quite well understood and many 
actions are known to produce spectral gaps (see \emph{e.g.} 
\cite{BoGa07,BeSa14,CoGu11,Sha00,Bek03,GJS99}). For example, if 
$a,b\in\SO(3,\overline{\QQ})$ 
are two matrices with algebraic coefficients and generate a non\=/abelian free group, 
then their action by rotations on $\SS^2$ has spectral gap \cite{BoGa07}. It follows that 
the level sets of the warped cone $\cone_{\angles{a,b}}\big(\SS^2\big)$ are uniformly 
quasi\=/isometric to a family of expander graphs.

In a more general setting, let $G$ be a compact simple Lie group,
$(g_1,\ldots,g_k)\in G^k$ random $k$\=/tuple and $\Gamma(g_1,\ldots,g_k)<G$ the 
generated subgroup. The following is an open conjecture (see \cite{GJS99}):

\begin{conj}
\label{conj:Gamburd.Jakobson.Sarnak}
 Let $k\geq 2$. For almost every $k$\=/tuple the action 
$\Gamma(g_1,\ldots,g_k)\curvearrowright G$ by left multiplication has a spectral gap.
\end{conj}

Using Theorem \ref{thm:intro:warped.cones.exp_iff_meas.exp} and the spectral criterion 
for expansion, we can restate this conjecture in the language of warped cones:

\begin{alphthm}
 Conjecture \ref{conj:Gamburd.Jakobson.Sarnak} holds true if and only if for almost every 
$k$\=/tuple with $k\geq 2$ one (any) unbounded sequence of level sets of the warped 
cone $\cone_{\Gamma(g_1,\ldots,g_n)}(G)$ forms a family of expanders.
\end{alphthm}

\

{\bf Organisation of the paper.}
In Section \ref{sec:notation.and.preliminaries} we 
introduce some basic facts and notation that we use throughout the paper and in 
Section \ref{sec:meas.exp.and.Cheeger.const} we prove Theorem 
\ref{thm:intro:approx.exp_iff_meas.exp}. 
In Section 
\ref{sec:metric.bounds.on.degrees} we describe reasonably flexible conditions 
under which it is possible to prove that approximating graphs of actions on metric spaces 
have uniformly bounded degrees. 
In Section \ref{sec:voronoi.tesselations} we introduce the Voronoi tessellations as a 
fairly general way of constructing regular partitions. These partitions will 
then be used in Section 
\ref{sec:exp.and.warped.cones}, where we define warped cones and we investigate their
relations with the approximating graphs, hence proving Theorem 
\ref{thm:intro:warped.cones.exp_iff_meas.exp}.

In Section \ref{sec:spectral.criterion} we prove the spectral criterion for measure 
preserving actions that characterises the expanding actions as those actions that have a 
spectral gap. 
This criterion is then used extensively in Section \ref{sec:finite.Kazhdan.sets}, where 
we explain how to adapt the formalism of Kazhdan sets to the setting of expanding 
actions. We then review some well\=/known results of representation theory and 
we provide numerous examples of expanding actions and other applications.

\

{\bf Acknowledgements.}
I wish to thank my supervisor Cornelia Druţu for her advice and encouragement and for 
carefully reading this work. I also wish to thank Yves Benoist for useful conversations, Gareth Wilkes for his helpful comments and the anonymous referee for pointing out several inaccuracies and mistakes in the previous version of this paper.

This work was funded by the EPSRC Grant 1502483 and the 
J.T.Hamilton Scholarship. The material is also based upon work supported by the NSF under 
Grant No. DMS-1440140 while the author was in residence at the MSRI in Berkeley during the 
Fall 2016 semester.

\section{Some notation and preliminary results}
\label{sec:notation.and.preliminaries}

Throughout the paper, $\CG$ will denote a simplicial non\=/oriented graph, possibly with 
unbounded degree (a graph has \emph{degree} $D$ if every vertex is contained in at most 
$D$ edges). We wish to remark that we only use graphs with unbounded degree for the sake 
of generality, so that the statements of the results in Section 
\ref{sec:meas.exp.and.Cheeger.const} do not require further hypotheses. All the other 
sections will only involve graphs with bounded degree.

For every graph $\CG$ and 
every set of vertices $W\subset \vertex 
(\CG)$ we will denote by 
$\partial W$ the \emph{external vertex boundary}
\[
 \partial W\coloneqq \big\{v\in \vertex(\CG)\smallsetminus W
    \bigmid \exists w\in W \text{ s.t. $(v,w)$ is an edge}\big\}.
\]
We define the Cheeger constant of the graph $\CG$ as the infimum
\[
 h(\CG)\coloneqq \inf\left\{\frac{\abs{\partial W}}{\abs{W}} \, \middle|
 \, W \subset \vertex(\CG) \text{ finite, } 
\abs{W}\leq\frac{1}{2}\abs{\vertex(\vphantom{\big|}\CG)} \right\}
\]
(if the graph is infinite the condition on the cardinality of the set $W$ is vacuous).

\begin{de}
\label{de:expander.graphs}
A sequence of finite graphs $\CG_n$ is a \emph{family of expanders} if 
$\abs{\vertex(\CG_n)}\to \infty$ and there are two constants $C,\varepsilon>0$ so that 
every graph $\CG_n$ has degree bounded above by $C$ and Cheeger constant $h(\CG_n)$ at 
least $\varepsilon$.
\end{de}

\begin{rmk}
 In the literature the Cheeger constant is usually defined using the edge\=/boundary, we 
chose to use this less standard definition because it is notationally more convenient 
for our purposes. The two notions are coarsely equivalent when dealing with graphs with 
bounded degree.
\end{rmk}

Most of our results will be better understood in the context of coarse geometry. Let 
$(X,d_X)$ and $(Y,d_Y)$ be metric spaces.

\begin{de}
A map $f\colon (X,d_X)\to(Y,d_Y)$ is a \emph{quasi\=/isometry} if there are two positive 
constants $L,A\geq 0$ so that for every couple $x,x'\in X$ we have
\[
 \frac{1}{L}d_X(x,x')-A\leq d_Y(f(x),f(x'))\leq Ld_X(x,x')+A
\]
and for every $y\in Y$ there exists a $x\in X$ with $d_Y\big(y,f(x)\big)\leq A$. 
Such a function is an \emph{$(L,A)$\=/quasi\=/isometry}. 
\end{de}

Two spaces are \emph{quasi\=/isometric} if there exists a quasi\=/isometry between them. 
It is easy to show that being quasi\=/isometric is an equivalence relation. 
More generally, two families of 
metric spaces $(X_n,d_{X_n})$, $(Y_n,d_{Y_n})$ are \emph{uniformly quasi-isometric} if 
there exists a family of quasi\=/isometries $f_n\colon (X_n,d_{X_n})\to(Y_n,d_{Y_n})$ 
which share the same constants $L$ and $A$.

Connected graphs can be seen as metric spaces when endowed with their path metric. The 
following Lemma is well known and it can be proved with easy but tedious counting 
arguments:

\begin{lem}
\label{lem:unif.qi.preserves.expanders}
If two families of graphs $\CX_n$ and $\CY_n$ with uniformly bounded degree are uniformly 
quasi\=/isometric, then one of them is a family of expanders if and only if the other is. 
\end{lem}

With an abuse of notation, we give the following:

\begin{de}
\label{de:expander.metric.spaces}
 A sequence of metric spaces $(X_n,d_n)$ \emph{forms a family 
of expanders} if it is uniformly quasi\=/isometric to a family of 
expander graphs $\CX_n$. 
\end{de}

\begin{rmk}
Note that Lemma \ref{lem:unif.qi.preserves.expanders} implies that a family of graphs 
with uniformly bounded degrees form a family of expanders in the sense of Definition 
\ref{de:expander.metric.spaces} if and only if it is a genuine family of expander graphs 
(Definition \ref{de:expander.graphs}).
\end{rmk}

One of the many remarkable features of expander graphs is that it is not possible to 
embed them into any Hilbert space without greatly distorting the metric. More precisely, 
recall the following:

\begin{de}
 A map between metric spaces $f\colon (X,d_X)\to(Y,d_Y)$ is a \emph{coarse embedding} if 
there are two increasing unbounded control functions $\eta_-,\eta_+\colon [0,\infty)\to [0,\infty)$ so that 
\[
 \eta_-\big(d_X(x,x')\big)\leq d_Y(f(x),f(x'))\leq  \eta_+\big(d_X(x,x')\big)
\]
for every $x,x'\in X$.
We say that a family of metric spaces $(X_n,d_{X_n})$ \emph{coarsely embeds} in $(Y,d_Y)$ 
if there are coarse embeddings $f_n\colon (X_n,d_{X_n})\to (Y,d_Y)$ all with the same 
control functions $\eta_-$ and $\eta_+$ (we are dropping the word `uniform' from the 
terminology).
\end{de}

It is a well known fact that a family of expanders does not coarsely embed into any $L^p$ 
space for $1\leq p<\infty$ \cite{Mat97} (see also \cite[Appendix A]{Oza04}). 

\

Throughout the paper we will denote 
by $\Gamma$ a finitely generated group and $S$ will always denote a finite symmetric 
generating set containing the identity element: $1\in S$, $S=S^{-1}$.

We will generally consider measurable 
actions of $\Gamma$ on a measure space $(X,\nu)$ (\emph{i.e.} actions such that $g\colon 
X\to X$ is measurable for every 
$g\in\Gamma$). For every subset $A\subseteq X$ we 
will always use the notation $S\cdot A$ to denote the union of the images of $A$ under 
the elements of $S$
\[
 S\cdot A\coloneqq \bigcup_{s\in S}s\cdot A.
 \]
The requirement $1\in S$ is mainly made out of convenience, so that we always have 
$A\subseteq S\cdot A$.

\section{Measure expansion and Cheeger constants}
\label{sec:meas.exp.and.Cheeger.const}

We begin this section by defining a notion of expansion for actions on measure spaces:

\begin{de}
 A measurable action 
$\rho\colon\Gamma\curvearrowright (X,\nu)$ is 
\emph{expanding in measure} if there exists a constant $\alpha>0$ such that 
$\nu(S\cdot A)\geq \big(1+\alpha\big)\nu(A)$ for every measurable set $A\subset X$ with 
finite measure and $\nu(A)\leq \nu(X)/2$ (the latter condition is vacuous when $X$ has 
infinite measure).
When this is the case, we say that the action is \emph{$\alpha$\=/expanding}.
\end{de}

\begin{rmk}
 Our notion of expansion is not quite new in the literature. In the notation of 
\cite{GMP16}, 
the action $\rho$ is expanding in measure if $X$ is a \emph{domain of expansion} for it. 
In \cite{BoYe13} such a $\rho$ is said to be a \emph{continuous expander} (if the action 
is also differentiable). 
\end{rmk}

We define a \emph{measurable partition} of $(X,\nu)$ to be a countable family of 
disjoint subsets (regions) $\CP=\big\{R_i\bigmid i\in I\big\}$ such that
\[
 \nu\left( X\smallsetminus \coprod_{i\in I}R_i\right)=0.
\]
We can now define the key object of our study:

\begin{de}
 Given a measurable action $\rho\colon\Gamma\curvearrowright X$ and a measurable 
partition 
$\CP=\{R_i\mid i\in I\}$ of $(X,\nu)$, their \emph{approximating graph} is the (non 
oriented) graph $\CG_\rho(\CP)$ whose set of vertices is the set of regions
$\vertex\big(\CG_\rho(\CP)\big)\coloneqq \CP$ and such that the couple $(R_i,R_j)$ is an 
edge if and only if there exists an element $s\in S$ with
\[
 \nu\big((s\cdot R_i)\cap R_j\big)\neq 0.
\]
When there is no risk of confusing the action $\rho$, we will drop it from the notation 
and simply denote the approximating graph by $\CG(\CP)$.
\end{de}

For the graph $\CG_\rho(\CP)$ to give any interesting information on the dynamical 
system, 
we need to require some tameness conditions on the action itself and on the partition 
considered. The most important among such requirements is some kind of control on the 
ratios of the measures of the regions of the partition $\CP$.

\begin{de}
 A partition $\CP$ has \emph{bounded measure ratios} if the measure of every region 
$R\in\CP$ is finite and there exists a constant $Q\geq 1$ such that for every couple of 
regions $R_i,R_j$ in $\CP$ one has
 \[
  \frac{1}{Q}\leq \frac{\nu(R_i)}{\nu(R_j)}\leq Q.
 \]
\end{de}

The fundamental observation is the following simple lemma:

\begin{lem}
\label{lem:meas.exp___cheeg.const}
 Let $\rho\colon\Gamma\curvearrowright (X,\nu)$ be an $\alpha$\=/expanding action. 
For every partition $\CP$ with measure ratios bounded by a constant $Q$, the 
approximating graph $\CG_\rho(\CP)$ has 
Cheeger constant bounded away from zero
\[
 h\big(\CG_\rho(\CP)\big)\geq \varepsilon >0
\]
and the constant $\varepsilon=\varepsilon(\alpha,Q)$ depends only on the expansion 
parameter $\alpha$ and the bound on measure ratios $Q$.
\end{lem}

\begin{proof}
 Let $W$ be any finite set of vertices of $\CG_\rho(\CP)$ with $\oldabs{W}\leq 
\abs{\CG_\rho(\CP)}/2$ and consider the measurable set 
\[
 A\coloneqq \smashoperator[r]{\coprod_{R_i\in W}} R_i.
\]
Up to measure zero sets we have 
\[
 S\cdot A \subseteq \bigcup\Big\{R_i\Bigmid \nu\big((S\cdot A)\cap R_i\big)>0\Big\}
    = \smashoperator[r]{\coprod_{W\cup\partial W}}R_i \eqqcolon B.
\]
If $\nu(A)\leq \nu(X)/2$ then
\[
 \nu(B)\geq (1+\alpha)\nu(A)
\]
and since $\CP$ has bounded ratios we conclude
\[
 Q\abs{\vphantom{\big|}\partial W}\left(\inf_{R\in\CP}\nu(R)\right)
    \geq \nu(B\smallsetminus A)\geq \alpha\nu(A)
  \geq \alpha\abs{\vphantom{\big|}W}\left(\inf_{R\in\CP}\nu(R)\right);
\]
whence
\[
 \frac{\abs{\partial W}}{\abs{W}}\geq\frac{\alpha}{Q}.
\]

On the other hand, if $\nu(A)> \nu(X)/2$ let $C\coloneqq X\smallsetminus (S\cdot A)$ and 
notice that $(S\cdot C)\smallsetminus C\subseteq (S\cdot A)\smallsetminus A$ because 
$S=S^{-1}$.
Then we have:
\[
 \nu\big((S\cdot A)\smallsetminus A \big)
  \geq \alpha \nu(C)=\alpha\Big(\nu(X)-\nu(A)-\nu\big((S\cdot A)\smallsetminus A 
\big)\Big)
\]
whence
\[
 \nu(B\smallsetminus A)
  \geq \nu\big((S\cdot A)\smallsetminus A \big)
  \geq \frac{\alpha}{1+\alpha}\big(\nu(X)-\nu(A)\big).
\]

Since $\abs{W}\leq \abs{\CP}/2$, by the bound on measure ratios we get
\[
 \nu\big(X\smallsetminus A\big)\geq \frac{1}{Q}\nu(A).
\]
Using the same argument as above and combining the inequalities so obtained we conclude 
that
\[
 \frac{\abs{\partial W}}{\abs{W}}\geq 
    \min\left\{ \frac{\alpha}{Q},
      \frac{\alpha}{(1+\alpha)Q^2} \right\}
\]
as desired.
\end{proof}

\begin{cor}
 Let $(X,\nu)$ be a probability space and $\Gamma\curvearrowright (X,\nu)$ an action 
expanding in measure. Assume we are given a sequence of finite measurable partitions 
$\CP_n$ with $\abs{\CP_n}\to\infty$ and measure ratios uniformly bounded by the same 
constant $Q$.
Then the sequence of approximating graphs $\CG(\CP_n)$ forms a family of expanders if and 
only if they have uniformly bounded degree.
\end{cor}

In certain situations, considering finer and finer measurable partitions on the same 
dynamical system $\Gamma\curvearrowright (X,\nu)$ can yield a converse to Lemma 
\ref{lem:meas.exp___cheeg.const}.
Recall that a measure $\nu$ on the Borel $\sigma$\=/algebra of a Hausdorff topological 
space is a \emph{Radon measure} if it is locally finite (every point has a neighbourhood 
of finite measure) and inner regular (the measure of every measurable set $A\subseteq X$ 
is equal to the supremum of the measures of its compact subsets $K\subseteq A$).

Let $X$ be also a metric space and recall that the diameter of a subset $A\subseteq X$ is defined as $\diam(A)\coloneqq\sup\{d(x,y)\mid x,y\in A\}$. We can prove the following:

\begin{prop}
\label{prop:cheeg.const___expanding}
Let $(X,d,\nu)$ be a locally compact metric space with a Radon measure thereof and let $\CP_n$ be a sequence of measurable partitions of $X$ with uniformly bounded measure ratios. Assume that for every compact set $K\subseteq X$ there is a decreasing sequence of positive numbers $(r_{K,n})_{n\in\NN}$ such that $r_{K,n}\to 0$ and $\nu\bigparen{{E}_{K,n} }\to 0$, where
\[
 E_{K,n}\coloneqq \bigcup\big\{R\bigmid R\in\CP_n,\ \nu({R}\cap K)>0,\ \diam(R)>r_{K,n}\big\}.
\]
Then, for any continuous action $\Gamma\curvearrowright X$ the existence of a uniform positive Cheeger constant $\epsilon>0$ for the approximating graphs $\CG\paren{\CP_n}$ implies that the action is expanding in measure.
\end{prop}

\begin{proof}
%
 Since $\nu$ is a Radon measure, it is enough to prove that there is a positive constant $\alpha>0$ so that for every compact set $K\subset X$ with finite measure and $\nu(K)\leq\nu(X)/2$ we have $\nu(S\cdot K)\geq(1+\alpha)\nu(K)$. 
 
Fix such a compact set $K$ and an appropriate sequence $r_{K,n}\to 0$ and note that by hypothesis we have: 
\[
 \nu(K)
  =\lim_{n\to\infty}\bigbrack{ \nu\big(K\smallsetminus {E}_{K,n}\big)
      +\nu\big(K\cap {E}_{K,n}\big)}
  =\lim_{n\to\infty}\nu\big(K\smallsetminus {E}_{K,n}\big).
\]

For any set $A\subseteq X$ we will denote by $\vertex_n(A)\subseteq \CP_n$ the set of cells in $\CP_n$ that intersect $A$ with positive measure $
 \vertex_n(A)\coloneqq \bigbraces{R\in\CP_n\bigmid \nu({R}\cap A)>0 }$
and denote by $\CN_n[A]$ their union $
 \CN_n[A]\coloneqq{\bigcup_{R\in \vertex_n(A)}}R \subseteq X$.

Since $E_{K,n}$ is a union of regions, for every $A\subseteq K$ we have that $\CN_n\brackets{A\smallsetminus {E}_{K,n}}= \CN_n\brackets{A}\smallsetminus E_{K,n}$. Moreover, since $A\subseteq K$, if $R$ is a region contained in $\CN_n\brackets{A}\smallsetminus E_{K,n}$ then $R$ is contained in the neighbourhood ${N}_{r_{K,n}}(A)$. We thus obtain: 
\begin{equation}
 \label{eq:CN.subset.small.nbhd}
 {\CN}_n\bigbrack{K \smallsetminus {E}_{K,n}}
 = { \CN_n\brackets{K} \smallsetminus E_{K,n} }\subseteq {{N}_{r_{K,n}}(K)\smallsetminus E_{K,n} }.
\end{equation}

For every $A\subseteq X$, we have $A\subseteq\CN_n\brackets{A}$ up to measure $0$ sets, we obtain:
\[
  \nu\big(K\smallsetminus {E}_{K,n}\big)
 \leq \nu\bigparen{
    \CN_n \bigbrack{K\smallsetminus {E}_{K,n}) } } 
 \leq\nu\big({N}_{r_{K,n}}(K)\big),
\]
and since both the first and the last expression tend to $\nu(K)$ we deduce that there exist the limit
\begin{equation}
 \label{eq:appgraph:measure.K.equal.limit}
 \lim_{n\to\infty}\nu\bigparen{\CN_n \bigbrack{K\smallsetminus {E}_{K,n} }}
 =\nu(K).
\end{equation}

\

Equation~\eqref{eq:appgraph:measure.K.equal.limit} allows us to express $\nu(K)$ as a limit of measures of finite unions of regions of $\CP_n$. We now need to find a similar estimate for $S\cdot K$. 
For $n$ large enough the closure of the neighbourhood ${N}_{r_{K,n}}\bigpar{K}$ is compact and hence for every $s\in S$ the restriction of $s$ to said compact neighbourhood is a uniformly continuous map. It follows that there is an infinitesimal decreasing sequence $r'_{K,n}\to 0$ such that $s\cdot {N}_{r_{K,n}}(K)\subseteq {N}_{r'_{K,n}}\bigpar{s\cdot K}$ for every $s\in S$ and hence
\begin{equation}
\label{eq:S.nbhd.contained.big.nbhd}
 S\cdot {N}_{r_{K,n}}\paren{K}
 \subseteq {N}_{r'_{K,n}}\bigpar{S\cdot K}.
\end{equation}

We can also fix a $n_0$ large enough so that the set $ C\coloneqq{N}_{r'_{K,n_0}}\bigpar{S\cdot K}$
is compact. In particular, we obtain a new infinitesimal sequence $r_{C,n}$ from the hypothesis and, as for \eqref{eq:CN.subset.small.nbhd}, we have
\begin{equation}
 \label{eq:CN.subset.small.nbhd.setC}
 \CN_n\brackets{A} \smallsetminus E_{C,n} 
 \subseteq {N}_{r_{C,n}}(A)\smallsetminus E_{C,n}
 \subseteq {N}_{r_{C,n}}(A)
\end{equation}
for every $A\subseteq C$. 
 For $n\geq n_0$, applying \eqref{eq:CN.subset.small.nbhd}, \eqref{eq:S.nbhd.contained.big.nbhd} and \eqref{eq:CN.subset.small.nbhd.setC}, we obtain a chain of containments:
 \begin{align*}
  \CN_n\bigbrack{S\cdot    
    {\CN}_n\bigbrack{K\smallsetminus{E}_{K,n}}}
  &\subseteq \CN_n\bigbrack{S\cdot {N}_{r_{K,n}}(K)} \\
  &\subseteq 
    \CN_n\bigbrack{{N}_{r'_{K,n}}\bigpar{S\cdot K}}\\
  &\subseteq 
    \Bigparen{ \CN_n\bigbrack{{N}_{r'_{K,n}}\bigpar{S\cdot K}}\smallsetminus E_{C,n} } \cup E_{C,n}\\
  &\subseteq 
    {N}_{r_{C,n}+r'_{K,n}}\bigpar{S\cdot K} \cup E_{C,n}.
 \end{align*}
 Since the measure of the last term converges to $\nu\bigparen{S\cdot K}$, we deduce
 \[
  \nu\bigparen{S\cdot K}\geq 
    \limsup_{n\to\infty} \nu\Bigparen{\CN_n\bigbrack{S\cdot {\CN}_n\bigbrack{K\smallsetminus{E}_{K,n}}} }
 \]
and together with \eqref{eq:appgraph:measure.K.equal.limit} this yields:
\begin{equation}\label{eq:appgraph:ratio.measures}
 \frac{\nu(S\cdot K)}{\nu(K)}
  \geq\limsup_{n\to\infty}
    \frac{\nu\Big(\CN_n\bigbrack{S\cdot {\CN}_n\bigbrack{K\smallsetminus{E}_{K,n}}}\Big)}
	  {\nu\Big(\CN_n\bigbrack{K\smallsetminus {E}_{K,n} }\Big)}.
\end{equation}

\

By hypothesis, the partitions $\CP_n$ have uniformly bounded measure ratios. That is, there exists a constant $Q$ such that for any $n$ and any $R,R'\in \CP_n$ we have $\nu(R)\leq Q\nu(R')$. It follows that for every pair of sets $A,B\subseteq X$ we have an estimate
\[
 \frac{\nu\bigparen{ \CN_n[A] }}{\nu\bigparen{\CN_n[B] }}
 \geq \frac{\abs{\vertex_n(A)}}{\abs{\vertex_n(B)}} Q^{-1}. 
\]

Now, the key point of the proof is that for any subset $A\subseteq X$ we have that $ \vertex_n\bigparen{S\cdot {\CN}_n[A]}$ coincides with $\vertex_n(A)\sqcup\partial\vertex_n(A)$. Thus we get 
\[
 \frac{\nu\big(\CN_n\bigbrack{S\cdot {\CN}_n[A]}\big)}{\nu\big(\CN_n[A]\big)}
  =1+\frac{\nu\Big(\CN_n[S\cdot {\CN}_n[A]]\smallsetminus\CN_n[A]\Big)}{\nu\big(\CN_n[A]\big)}
  \geq 1+ \frac{\abs{\vphantom{\big|}\partial\vertex_n(A)}}
		{\abs{\vertex_n(A)}}Q^{-1}
\]
and if we apply this inequality to the sets $A_n\coloneqq K\smallsetminus {E}_{K,n}$, inequality \eqref{eq:appgraph:ratio.measures} becomes
\[
  \frac{\nu(S\cdot K)}{\nu(K)}
  \geq 1+ \limsup_{n\to\infty}
    \frac{\abs{\vphantom{\big|}\partial\vertex_n(A_n\big)}}
		{\abs{\vertex_n(A_n)}}Q^{-1}.
\]
It is hence enough to find a uniform bound $\alpha>0$ such that
\[
 \limsup_{n\to\infty}\frac{\abs{\vphantom{\big|}\partial\vertex_n(A_n)}}
		{\abs{\vertex_n(A_n)}}\geq \alpha.
\]

Note that for large $n$ the set ${\vertex_n(A_n)}$ si finite because $\CN[A_n]\subseteq N_{r_{K,n}}(K)$ has finite measure. If $\abs{\vertex_n(A_n)}$ is less than or equal to $\abs{\vertex_n(X)}/2$ then $
 \frac{\abs{\partial\vertex_n(A_n)}}
		{\abs{\vertex_n(A_n)}}\geq \epsilon$
by definition of Cheeger constant. If this is not the case, we need to use an argument similar to that of Lemma \ref{lem:meas.exp___cheeg.const}. That is, denote by $W_n$ the complement set $W_n\coloneqq\vertex_n(X)\smallsetminus\vertex_n(A_n)$ and notice that $\partial\vertex_n(A_n)
    =\partial_{\text{int}}(W_n)
    \supseteq\partial\big(W_n\smallsetminus\partial_{\text{int}}(W_n)\big)$,
where $\partial_{\rm int} (W_n)$ denotes the interior vertex boundary, \emph{i.e} the set of vertices of $W_n$ which are endpoints of edges with one endpoint in the complement $\vertex_n(X)\smallsetminus W_n$. It follows that $
 \abs{\partial\vertex_n(A_n)}
    \geq \epsilon\abs{\vphantom{\big(}W_n\smallsetminus\partial_{\text{int}}(W_n)}
    =\epsilon \big(\abs{ W_n}-\abs{\partial\vertex_n(A_n)}\big)$.

Since $\lim_{n\to \infty}\nu\big(\CN_n(A_n)\big)=\nu(K)\leq \frac{1}{2}\nu(X)$, we can once more use the bound on measure ratios to get that $\abs{\vertex_n(A_n)}\leq Q\abs{W_n}$ for $n$ large enough. Thus we have
\begin{align*}
 \limsup_{n\to\infty}\frac{\abs{\partial\vertex_n(A_n)}}{\abs{\vertex_n(A_n)}}
    &\geq \limsup_{n\to\infty}\Bigpar{ \frac{1}{\abs{\vertex_n(A_n)}}
	\min\left\{\epsilon\abs{\vertex_n(A_n)}\mathrel{,} \frac{\epsilon}{1+\epsilon}\abs{ W_n} \right\} }\\
    &\geq \limsup_{n\to\infty} \min\left\{\epsilon\mathrel{,} \frac{\epsilon}{Q(1+\epsilon)}\right\} 
	=\min\left\{\epsilon\mathrel{,} \frac{\epsilon}{Q(1+\epsilon)}\right\} 
\end{align*}
as desired.
\end{proof}

If $\CP$ is a partition of a metric space, its \emph{mesh} is is the supremum of the diameters of its regions:
\[
 {\rm mesh}(\CP)\coloneqq \sup\{\diam(R)\mid R\in\CP\}.
\]
Proposition \ref{prop:cheeg.const___expanding} immediately implies the following:

\begin{cor}\label{cor:cheeg.const.and.mesh___expanding}
 Given a continuous action $\Gamma\curvearrowright X$ on a locally compact metric space equipped with a Radon measure and a sequence of measurable partitions $\CP_n$ with uniformly bounded measure ratios and ${\rm mesh}(\CP_n)\to 0$; if the approximating graphs $\CG(\CP_n)$ have Cheeger constant uniformly bounded away from $0$ then the action is expanding in measure.
\end{cor}

\begin{rmk}
 It is easier to prove Corollary \ref{cor:cheeg.const.and.mesh___expanding} directly rather than proving Proposition \ref{prop:cheeg.const___expanding}, but we decided to provide a more general statement that could be applied \emph{e.g.} to metric spaces with cusps as well.
\end{rmk}

We can combine Proposition \ref{prop:cheeg.const___expanding} with Lemma 
\ref{lem:meas.exp___cheeg.const} to prove Theorem 
\ref{thm:intro:approx.exp_iff_meas.exp}. Specifically, we have the following: 

\begin{thm}
\label{thm:meas.exp_iff_unif.Cheeger.const}
 Let $\Gamma\curvearrowright (X,d,\nu)$ be a continuous action on a locally compact space equipped with a Radon measure and $\CP_n$ a family of measurable partitions of $X$ with uniformly bounded measure ratios.
Assume that for every compact set $K\subseteq X$ there is a decreasing sequence of positive numbers $(r_{K,n})_{n\in\NN}$ such that $r_{K,n}\to 0$ and $\nu\bigparen{{E}_{K,n} }\to 0$, where
\[
 E_{K,n}\coloneqq \bigcup\big\{R\bigmid R\in\CP_n,\ \nu({R}\cap K)>0,\ \diam(R)>r_{K,n}\big\}.
\]
Then, the action is expanding in measure if and only 
if all the approximating graphs $\CG(\CP_n)$ share a common lower bound on their Cheeger constant.
\end{thm}

\begin{rmk}
 Notice that if the approximating graphs are already known to share a uniform bound 
on their degrees, then one can modify the proofs of Lemma 
\ref{lem:meas.exp___cheeg.const} and Proposition \ref{prop:cheeg.const___expanding} in 
order to extend them to measurable 
actions of semigroups.
\end{rmk}

\section{Metric bounds on degrees}
\label{sec:metric.bounds.on.degrees}

Using the tools of Section \ref{sec:meas.exp.and.Cheeger.const} we can now build families 
of graphs with uniform lower bounds on their Cheeger constants. Still, to build examples 
of expanders we also need to find a way to bound the degrees of the 
approximating graphs. One way of doing it is by using additional metric structures.

Let $(X,d,\nu)$ be a metric space with a Borel measure $\nu$. Such measure is 
\emph{doubling} if there exists a constant $D$ such that 
\[
 \nu\big( B_x(2r)\big)\leq D\nu\big(B_x(r)\big)
\]
for every $x\in X$ and every radius $r>0$.

For any bounded subset $A\subset (X,d)$ we will define its \emph{eccentricity} as 
\[
 \xi(A)\coloneqq\inf\left\{\frac{R}{r} \, \middle |  \,
    \exists x \in X,\ B_x(r)\subseteq A\subseteq B_x(R)\right\}.
\]

\

Recall that, given an increasing homeomorphism $\eta\colon [0,+\infty)\to 
[0,+\infty)$, a homeomorphism $f\colon X\to X$ is \emph{$\eta$\=/quasi\=/symmetric} 
if 
\[
 \frac{d\big(f(x),f(y)\big)}{d\big(f(x),f(z)\big)}
 \leq \eta\left(\frac{d(x,y)}{d(x,z)}\right)
\]
for every choice of points $z\neq x\neq y$ in $X$. We say that $f$ is 
\emph{quasi\=/symmetric} if it is $\eta$\=/quasi\=/symmetric for some 
$\eta\colon[0,\infty)\to[0,\infty)$.

We will also say that a measurable map $f\colon (X,\nu)\to(X,\nu)$ has \emph{measure 
distortion} bounded by $\Theta\geq 1$ if 
\[
 \frac{1}{\Theta}\nu(A)\leq\nu\big(f(A)\big)\leq \Theta \nu(A)
\]
 for every measurable set $A\subseteq X$.
 
An action $\rho\colon G\curvearrowright (X,d)$ is quasi\=/symmetric (respectively, has 
bounded measure distortion) if $\rho(g)$ is a quasi\=/symmetric homeomorphism 
(respectively, has bounded 
measure distortion) for every $g\in G$. We do not require the bounds to be uniform.
 
\begin{prop}
\label{prop:nice.action.on.doubling___bded.degree}
 Let $(X,d,\nu)$ be a doubling measure space and $\CP$ be a measurable partition with 
measure ratios bounded by $Q\geq 1$ and such that all the regions $R\in \CP$ are bounded 
and have uniformly bounded eccentricity $\xi(R)\leq \xi$ for some constant $\xi>0$.

Given a quasi\=/symmetric action $\rho\colon\Gamma\curvearrowright (X,d,\nu)$ with 
bounded measure distortion,  
let $\eta\colon[0,\infty) \to[0,\infty)$ and $\Theta\geq 1$ be such that for every 
generator $s\in S$, the homeomorphism $\rho(s)$ is $\eta$\=/quasi\=/symmetric and has 
measure distortion bounded by $\Theta$.

Then the approximating graph $\CG(\CP)$ has bounded degree and this bound depends only on 
$\eta,\Theta,Q,\xi$ and the doubling constant $D$.
\end{prop}

\begin{proof}
 Fix any $s\in S$ and, for any region $R\in \CP$, let 
 \[
  J\coloneqq \big\{R_i\in\CP\bigmid R_i\cap s(R)\neq \emptyset\big\}.
 \]
It is enough to prove that there is a uniform bound on $\abs{J}$.

Since $s$ is $\eta$\=/quasi\=/symmetric, the image $s(R)$ has eccentricity at most 
$\eta(\xi)$. 
Thus there are $x\in X$ and $0< r_1\leq r_2$ with $r_2\leq \eta(\xi) r_1$ such that 
\[
 B_x(r_1)\subseteq s(R)\subseteq B_x(r_2).
\]
We can then bound the measure $\nu\big(B_x(r_2)\big)$ due to the doubling condition on $X$
\[
 \nu\big(B_x(r_2)\big)
  \leq D^{\left\lceil\log_2\left(r_2/r_1\right)\right\rceil}\nu\big(B_x(r_1)\big)
  \leq D^{\left\lceil\log_2\left(\eta(\xi)\right)\right\rceil}\nu\big(s(R)\big)
\]
where $\lceil t\rceil$ denotes the smallest integer $k\geq t$.
Letting $L_1=D^{\left\lceil\log_2\left(\eta(\xi)\right)\right\rceil}$ we get
\[
 \nu\big(B_x(r_2)\big)
  \leq L_1 \Theta\nu(R)
\]
by bounded measure distortion.

Choose a finite subset $J'\subseteq J$ (a priori, $J$ could still be infinite at this point). Let $r\geq 0$ be the smallest real number such that $R_i\subseteq B_x(r_2+r)$ for every region
$R_i\in J'$. By definition, there exists a region $R_j$ intersecting $s(R)$ non trivially 
and having diameter $\diam(R_j)\geq r$. Thus there exists $y\in X$ with 
$B_y(r/2\xi)\subseteq R_j$.

Notice that
\[
 \nu\left( \coprod_{i\in J'}R_i\right)
 \leq \nu\big(B_y (2(r+r_2))\big)
 \leq D^{\left\lceil\log_2\left(\frac{2(r+r_2)2\xi}{r}\right)\right\rceil}
      \nu\big(B_y(r/\xi)\big).
\]
Letting $L_2(t)\coloneqq D^{\left\lceil\log_2\left(4\xi(1+t)\right)\right\rceil}$ yields
\begin{equation}\label{eq:bound.on.J'}
 \frac{\abs{J'}}{Q}\nu(R_j)
 \leq \nu\big(B_y(2r+r_2)\big)
 \leq L_2(r_2/r)\nu\big(B_y(r/\xi)\big)
 \leq L_2(r_2/r)\nu(R_j),
\end{equation}
where the RHS is set as $+\infty$ if $r=0$. As the RHS is a decreasing function of $r$, and $r$ increases as $J'\subset J$ does, we deduce that there exists a nummber $r_3\geq 0$ which is the smallest such that $R_i\subseteq B_x(r_2+r_3)$ for every region
$R_i\in J$. Inequality~\eqref{eq:bound.on.J'} then implies
\[
 \frac{\abs{J}}{Q}  \leq L_2\Big(\frac{r_2}{r_3}\Big).
\]

\

On the other hand we have
\[
 \nu\left( \coprod_{i\in J}R_i\right)
 \leq \nu\big(B_x(r_3+r_2)\big)
 \leq D^{\left\lceil\log_2\left(\frac{(r_3+r_2)}{r_2}\right)\right\rceil}
      \nu\big(B_x(r_2)\big),
\]
thus letting $L_3(t)=D^{\left\lceil\log_2\left(t+1\right)\right\rceil}$ we get
\[
 \frac{\abs{J}}{Q}\nu(R)
 \leq L_3(r_3/r_2)\nu\big(B_x(r_2)\big)
 \leq L_3(r_3/r_2)L_1\Theta\nu(R).
\]
Thus we conclude
\begin{align*}
  \abs{J} &\leq \min\left\{ \vphantom{\Big |} QL_2\big(r_2/r_3\big), 
    Q\Theta L_3\big(r_3/r_2\big)L_1\right\} \\
  &\leq \sup_{t>0} \left( 
    \min\left\{ QL_2\big(t\big), Q\Theta L_3\left(\frac{1}{t}\right)L_1\right\}\right)
\end{align*}
and the latter is bounded.

\end{proof}

Proposition \ref{prop:nice.action.on.doubling___bded.degree} is fairly general in that it 
deals with a large class of maps and measure spaces. Still, at times its hypotheses might 
be cumbersome to work with. It is especially so when one already knows that the 
action and 
the tessellation already satisfy much stronger hypotheses that make some of the 
requirements of Proposition \ref{prop:nice.action.on.doubling___bded.degree} redundant. 
For example, it is easy to prove the following:

\begin{rmk}
\label{rmk:Lipschitz.on.nice.part___bded.degree}
 Let $(X,d)$ be a metric space with a partition $\CP$ so that there exists a function 
$\zeta\colon [0,\infty)\to \NN$ such that for every point $x\in X$ and every radius $r$ 
the ball $B_x(r)$ intersects at most $\zeta(r)$ regions of $\CP$.

Then, if $f\colon X\to X$ is a $L$\=/Lipschitz map and $Y\subset X$ is any subset,
then $f(Y)$ intersect at most $\zeta\big(L\diam(Y)/2\big)$ regions of $\CP$.
\end{rmk}

\section{Voronoi tessellations}
\label{sec:voronoi.tesselations}

A convenient way for defining measurable partitions on a metric space $(X,d)$ is given by 
the Voronoi tessellations. 

\begin{de}
 Given a countable discrete subset $Y\subseteq X$
the associated \emph{Voronoi tessellation} is the family $\CV(Y)\coloneqq \{R(y)\mid y\in 
Y\}$ where
\[
 R(y)\coloneqq \big\{x\in X
    \bigmid d(x,y)<d(x,y')\text{ for all }y'\in Y,\ y'\neq y\big\}.
\]
\end{de}

If the metric is nice enough, the regions $R(y)$ will be disjoint open sets. Assume that 
$X$ is also given a measure $\nu$ defined on the Borel subsets. In case that
for any couple of points $y\neq y'$ of $Y$ the hyperplane $P(y,y')=\{x\in X\mid 
d(x,y)=d(x,y')\}$ has measure zero, the Voronoi tessellation $\CV(Y)$ covers a conull 
subset of $X$ and it is therefore a measurable partition of $X$.

Typical examples of well\=/behaved metric spaces are the
Riemannian manifolds. Let $(M,\varrho)$ be a Riemannian manifold. The volume form 
induced by $\varrho$ naturally defines a measure $\nu$ on $M$ and every Voronoi 
tessellation of $M$ forms a measurable partition of $(M,\nu)$. Before we proceed, we 
fix some notation: 

\begin{de}
 Let $r>0$ be a constant. A subset $Y$ of a metric space $(X,d)$ is 
\emph{$r$\=/separated} if $d(y,y')\geq r$ for every couple of distinct 
points $y\neq y'\in Y$. It is \emph{$r$\=/dense} if for every point $x\in X$ there exists 
a $y\in Y$ with $d(x,y)<r$.
\end{de}

Recall that the \emph{injectivity radius} $\injrad(M)$ of a Riemannian manifold 
$(M,\varrho)$ is the largest $r\geq 0$ so that for every point $x\in M$ the exponential 
map is injective when restricted to the ball $B_0(r)\subset T_x(r)$. We will need the 
following technical lemma:

\begin{lem}
\label{lem:voronoi.has.bded.ratios}
 Let $(M,\varrho)$ be a complete $n$\=/dimensional Riemannian manifold with pinched 
sectional curvature $-K\leq \kappa_\varrho\leq K$ and positive injectivity radius 
$\injrad(M)>0$. Fix any constant $0<C<2\injrad(M)$. If a discrete subset 
$Y\subseteq M$ is $r$\=/separated and $R$\=/dense for some constants $0<r\leq R$ with 
$r\leq C$,
then the Voronoi tessellation $\CV(Y)$ has measure ratios bounded by $Q$, where 
$Q=Q(n,K,C,R/r)$ is a constant depending only on $n,K,C$ and the ratio $R/r$. 
\end{lem}

To prove Lemma \ref{lem:voronoi.has.bded.ratios} it is enough to realise that every 
Voronoi cell contains a ball of radius $r/2$ and is contained in a ball of radius $R$. 
One can then estimate the volume of such cell using the Bishop\=/Gromov Comparison 
Theorem. We will omit the details of the argument.

\begin{rmk}
 With the same techniques one can also prove 
that every compact Riemannian manifold is a doubling measure space.
\end{rmk}

We remark that in any metric space for any fixed $r>0$ we can find a maximal 
$r$\=/separated subset. Such subsets are exactly those sets which are both 
$r$\=/separated and $r$\=/dense. Then we have:

\begin{thm}
\label{thm:action.on.mfld.exp_iff_approx.graph.expand}
Let $\Gamma$ be a finitely generated group and $S$ a finite symmetric generating set 
which contains the identity.
 Given a compact Riemannian manifold $(M,\varrho)$ and a sequence $r_n\to 0$, choose for 
every $n$ a maximal $r_n$\=/separated set $Y_n$ and consider the associated Voronoi 
tessellation $\CV(Y_n)$. 
If $\Gamma$ acts by quasi\=/symmetric homeomorphisms on $M$ and the action has bounded 
measure distortion, then the sequence of 
approximating graphs $\CG\big(\CV(Y_n)\big)$ is a family of expanders if and only if the 
action is expanding in measure.
\end{thm}

\begin{proof}
Since $M$ is compact, $\injrad(M)>0$ and $M$ has pinched sectional curvature.  
For all but finitely many $n$ we have $r_n<\injrad(n)$ and we can therefore apply Lemma 
\ref{lem:voronoi.has.bded.ratios} to deduce that the measurable 
partitions $\CV(Y_n)$ have uniformly bounded measure ratios. Moreover, the 
diameter of the regions $R\in\CV(Y_n)$ is at most $2r_n$ because 
$Y_n$ is a maximal $r_n$ separated set. 

We are now under the hypotheses of Theorem \ref{thm:meas.exp_iff_unif.Cheeger.const}, 
whence we deduce that the 
approximating graphs $\CG\big(\CV(Y_n)\big)$ have a uniform lower bound on their Cheeger 
constant if and only if the action of $\Gamma$ is expanding in measure.

The regions of the Voronoi tessellation $\CV(Y)$ have at most eccentricity $2$ and we 
already remarked that $(M,\nu)$ is a doubling measure space. We are then under the 
hypotheses of Proposition \ref{prop:nice.action.on.doubling___bded.degree} and thus we 
conclude that all the graphs $\CG\big(\CV(Y_n)\big)$ have uniformly bounded degrees.
\end{proof}

\section{Expanders, warped cones and coarse embeddings}
\label{sec:exp.and.warped.cones}

Let $(X,d)$ be a compact metric space and let $\Gamma=\angles{S}$ act on $X$ by 
homeomorphisms. Following \cite{Roe05}, we define the \emph{warped metric} $\wdist{}$ on 
$X$ as the maximal distance such that 
$\wdist{}(x,y)\leq \dist{}(x,y)$ and $\wdist{} (x,sx)\leq 1$ for every $x,y\in X$ 
and $s\in S$.

\begin{rmk}
 \label{rmk:expression.for.wdist}
 One can show that given two points $x,y\in X$ their warped distance $\wdist{}(x,y)$ is 
equal to the infimum of the sums
\[
 \dist{}(x_0,x_1)+1+\dist{}(s_1\cdot x_1,x_2)+1+\cdots+\dist{}(s_{n-1}\cdot 
x_{n-1},x_n)
\]
where the $x_i$'s are points in $X$ with $x=x_0$ and $y=x_n$ and the 
$s_i$'s are elements in $S$. Since $X$ is locally compact, this infimum is actually a 
minimum.
\end{rmk}

For simplicity, we will restrict our attention to metric spaces of the form $(M,d)$ where 
$M$ is a compact Riemannian manifold and $d$ is the distance induced by the Riemannian 
metric $\varrho$. 
A similar construction can be carried out for other well\=/behaved compact metric spaces 
such as finite simplicial complexes. 

For $t\geq 1$, we consider the family of metric spaces $(M,\dist{t})$ obtained rescaling 
the distance $\dist{t}\coloneqq t\cdot \dist{}$. We can then warp each of these spaces 
under the action of $\Gamma$ obtaining a different family $(M,\wdist{t})$.
Notice that $\dist{t}(x,y)$ is equal to $\wdist{t}(x,y)$ whenever one of the two 
(and hence both) is smaller than $1$. It follows that a set $Y\subset M$ is 
$1/3$\=/separated in $(M,\dist{t})$ if and only if so it is in $(M,\wdist{t})$.

For every $t\geq 1$, let $Y_t\subset(M,\wdist{t})$ be a maximal $1/3$\=/separated set and 
let $\CX_t$ be the simplicial graph whose set of vertices is $Y_t$ and whose edges are 
the 
pairs of vertices with warped distance less than $2$:
\[
 E(\CX_t)=\big\{(x,y)\bigmid x,y\in Y_t,\ \wdist{t}(x,y)< 2\big\}.
\]
The graphs $\CX_t$ uniformly represent the coarse structure of the metric spaces $(M, 
\wdist{t})$. In fact, it is easy to verify that the inclusion of the vertices 
$Y_t\subset M$ induces a $(2,1)$\=/quasi\=/isometry between 
the graph $\CX_t$ and $(M,\wdist{t})$ for every $t\geq 1$.

\

Consider now the Voronoi tessellations $\CV(Y_t)$ of $M$. We would like to say that the 
graph $\CG\big(\CV(Y_t)\big)$ approximating the action $\Gamma\curvearrowright M$ is 
in some sense the same graph as $\CX_t$. This is not quite the case, because two vertices 
$y,y'\in Y_t$ may be far away in the approximating graph $\CG\big(\CV(Y_t)\big)$ but quite 
close under the distance $\dist{t}$ and thus form an edge of $\CX_t$. For convenience we 
will then define the \emph{lavish approximating graphs} as the graphs $\widetilde 
\CG\big(\CV(Y_t)\big)$ whose vertices are the elements of $Y_t$ and such that $(y,y')$ 
forms an edge if and only if
\[
 s\cdot \overline{R(y)}\cap \overline{R(y')}\neq \emptyset
\]
for some $s\in S$.

The lavish approximating graph could contain many more edges then the usual 
approximating graph and it could also have a rather different geometry.
Despite this, the lavish approximating graphs still encode well 
the dynamic properties of the action of $\Gamma$. 

\begin{lem}
 \label{lem:lavish.graph.expanders_iff_measure.expand}
 Let $\Gamma$ act on a compact Riemannian manifold $(M,\varrho)$ \emph{via} 
quasi\=/symmetric 
homeomorphisms and with bounded measure distortion. 
For any sequence $t_n>1$ with $t_n\to \infty$ the action of 
$\Gamma$ is expanding in measure if and only if the lavish approximating graphs 
$\widetilde \CG\big(\CV(Y_{t_n})\big)$ form a family of expanders.
\end{lem}

\begin{proof}
 Notice first that the Voronoi tessellations obtained by the set $Y_n$ with respect to 
the 
metrics $d$ and $\dist{t}$ are actually the same. Hence they yield the 
same approximating graphs. On $(M,d)$, the set $Y_n$ is a maximal $(1/3)t_n^{-1}$ 
separated set and thus we are in the hypotheses of Theorem 
\ref{thm:action.on.mfld.exp_iff_approx.graph.expand}. From this we deduce that the usual 
approximating graphs $ \CG\big(\CV(Y_{t_n})\big)$ form a family of expanders if and only 
if the 
action is expanding in measure.

Now, the graph $\widetilde \CG\big(\CV(Y_{t_n})\big)$ contains the approximating graph 
$\CG\big(\CV(Y_{t_n})\big)$ and they both have the same set of vertices, thus the Cheeger 
constants of 
the lavish approximating graphs are at least as big as those of the ordinary 
approximating graphs. The graphs $\widetilde \CG\big(\CV(Y_{t_n})\big)$ have uniformly 
bounded degree because 
there is a uniform bound on the extra edges that are 
added to each vertex (the number of adjacent tiles can be estimated using the doubling 
condition and the uniform bound on the eccentricity). It follows that when the action 
is expanding in measure, the graphs 
$\widetilde \CG\big(\CV(Y_{t_n})\big)$ form a family of expanders.

For the other implication it is enough to notice that for any vertex $y\in 
Y_{t_n}$ we have that the union of the adjacent regions in the lavish graph
\[
 A\coloneqq\bigcup\big\{R(y')\bigmid (y,y')\text{ is an edge in }\widetilde 
\CG\big(\CV(Y_{t_n})\big)\big\}
\]
is actually contained in a neighbourhood of radius $3t_n^{-1}$ of the union of the 
adjacent regions in the usual graph
\[
 A\subseteq N_{3t_n^{-1}}\left(
 \bigcup\big\{R(y')\bigmid (y,y')\text{ is an edge in }\CG\big(\CV(Y_{t_n})\big)\big\}
 \right).
\]
Then the same proof as for Proposition \ref{prop:cheeg.const___expanding} implies that if 
the lavish approximating graphs form a family of expanders then the action of $\Gamma$ is 
expanding in measure.
\end{proof}

The convenience of the lavish approximating graphs comes from the following:

\begin{lem}
\label{lem:lavish.approx.graph_q.i_Xt}
 For every $t\geq 1$ the lavish approximating graph $\widetilde \CG\big(\CV(Y_t)\big)$ is 
contained 
in the graph $\CX_t$ and the inclusion is a $(L,A)$\=/quasi\=/isometry where the 
constants $L$ and $A$ depend only on the geometry of $M$.
\end{lem}

\begin{proof}
Both $\widetilde \CG\big(\CV(Y_t)\big)$ and $\CX_t$ have $Y_t$ as set of vertices. If 
$(y,y')$ is an 
edge in $\widetilde \CG\big(\CV(Y_t)\big)$, by definition there must exist an element 
$x\in X$ such 
that $x\in \overline{R(y)}$ and $s\cdot x \in \overline{R(y')}$. Then
\[
 \wdist{t}(y,y')\leq \wdist{t}(y,x)+\wdist{t}(x,s\cdot x)+\wdist{t}(s\cdot x, y')<2
\]
thus $(y,y')$ is also an edge of $\CX_t$.

Conversely, if $\wdist{t}(y,y')<2$ then either we also have $\dist{t}(y,y')<2$ or there 
exist a point $x\in X$ with $\dist{t}(y,x)<1$ and $\dist{t}(s\cdot x,y')<1$. It is hence 
enough to bound the distance in $\widetilde \CG\big(\CV(Y_t)\big)$ of two vertices $y,y'$ 
with 
$\dist{t}(y,y')<2$.

Picking a geodesic path $\gamma$ in $(M,\dist{t})$ between $y$ and $y'$ we can define a 
sequence of vertices $y=y_0,y_1,\ldots,y_n=y'$ by keeping track of which regions of 
$\CV(Y_t)$ are traversed by $\gamma$. Then each couple $(y_i,y_{i+1})$ is an edge of 
$\widetilde \CG\big(\CV(Y_t)\big)$. We can bound $n$ using the geometry of $M$ because 
all 
the regions $R(y_i)$ are contained in a ball of radius $3$ of $(M,\dist{t})$ and thus one 
can obtain the required uniform bound using volume estimate techniques (it is the same 
kind of argument needed to prove Lemma \ref{lem:voronoi.has.bded.ratios}).
\end{proof}

Following \cite{Roe05} we will now define the warped cones. Let $(M,\varrho)$ be a 
compact Riemannian manifold, we define the open cone on $M$ as the space $\cone 
(M)\coloneqq M\times [1,\infty)$ with the metric $d_\cone$ induced by the Riemannian 
metric 
$\varrho_\cone\coloneqq t^2\varrho+dt^2$ 
where $dt^2$ is the standard Euclidean metric on $\RR$. An action of $\Gamma$ on $M$ by 
homeomorphism induces an action by homeomorphisms on the cone $\cone (M)$ which fixes the
coordinate $t$.

\begin{de}
 The \emph{warped cone} $\cone_\Gamma (M)$ of the manifold $M$ under the action of 
$\Gamma$ 
is the 
metric space $\big(\cone (M),\wdist{}\big)$ where $\wdist{}$ is the metric obtained 
warping the 
cone 
metric $d_\cone$ of $\cone (M)$.
\end{de}

In the previous part of this section, we have been studying the graphs $\CX_t$ which were 
quasi\=/isometric approximations of the metric spaces $(M,\wdist{t})$. Now, one would be 
temped to 
see them as approximations of the geometry of the level sets of the warped cone 
$\cone_\Gamma (M)$, and indeed we have the following:

\begin{lem}
\label{lem:level.sets_q.i_warped.spaces}
 For every $t_0\geq 1$, the level set $M\times\{t_0\}$ with the restriction of the metric $\wdist{}$ is 
$C$\=/bi\=/Lipschitz equivalent to the metric space $(M,\wdist{t_0})$ with a constant $C$ 
depending only on the geometry of $M$.
\end{lem}

\begin{proof}
 The obvious inclusion $(M,\wdist{t_0})\hookrightarrow (M\times \{t_0\},\wdist{})$ is a 
$1$\=/Lipschitz map. Indeed, for every two points $x,y\in M$ there exist a sequence 
$\gamma_0,\ldots,\gamma_n$ of geodesic paths $\gamma_i\colon[0,1]\to(M,\dist{t_0})$ such 
that $\gamma_0(0)=x$, $\gamma_n(1)=y$, $\gamma_{i+1}(0)=s_i\cdot\gamma_i(1)$ for some 
$s_i\in S$ and
\[
 \wdist{t_0}(x,y)=l(\gamma_0)+1+l(\gamma_1)+1+\cdots+l(\gamma_n).
\]

Since the metric $\dist{t_0}$ is trivially equal to the metric obtained \emph{via} the 
pull-back 
of the Riemannian metric $t^2\varrho+dt^2$, the paths $\gamma_i$ have the same length 
with respect to the metric $\dist{}_\cone$. The action of $\Gamma$ is the same and 
it still costs $1$ to `jump' from $z$ to $s\cdot z$. Thus the path 
$\gamma_0,\ldots,\gamma_n$ has the same length in $\big(\cone_\Gamma (M),\wdist{}\big)$ 
and hence 
$\wdist{}(x,y)\leq\wdist{t_0}(x,y)$.

Conversely, for every $x,y\in M\times\{t_0\}$ there exists a sequence 
$\gamma_0,\ldots,\gamma_n$ of geodesic paths $\gamma_i\colon[0,1]\to\big(\cone_\Gamma 
(M),\wdist{}\big)$ such 
that $\gamma_0(0)=x$, $\gamma_n(1)=y$, $\gamma_{i+1}(0)=s\cdot\gamma_i(1)$ for some 
$s\in S$ and
\[
 \wdist{}(x,y)=l(\gamma_0)+1+l(\gamma_1)+1+\cdots+l(\gamma_n).
\]

Let $\widetilde\gamma_i$ the projection of the path $\gamma_i$ into the level set 
$M\times\{t_0\}$. Now the path $\widetilde\gamma_0,\ldots,\widetilde\gamma_n$ also joins 
the points $x,y$ in $(M,\wdist{t_0})$, thus we have
\begin{equation}
\label{eq:wdist(x,y)}
 \wdist{t_0}(x,y)
    \leq l(\widetilde\gamma_0)+1+l(\widetilde\gamma_1)+1+\cdots+l(\widetilde\gamma_n).
\end{equation}
It is hence enough to bound the lengths $l(\widetilde\gamma_i)$. 

For $i=0,\ldots,n$, let $t_i\geq 1$ denote the level of the starting point of the path 
$\gamma_i$, so that $\gamma_i(0)$ and $\gamma_{i-1}(1)$ lie in $M\times\{t_i\}$. Let 
$p\colon \big(\cone (M),\dist{}_\cone\big)\to 
\big(M\times\{1\},\dist{}_\cone\big)\cong (M,d)$ 
denote the projection.
Since the paths $\gamma_i$ are geodesics on $\big(\cone (M),\dist{}_\cone\big)$, the 
images of their 
projections $p(\gamma_i)$ are geodesic paths in $(M,d)$ (possibly with non\=/uniform 
speed). Denote by $\widehat\gamma_i\colon [0,\ell_i]\to (M,\varrho)$ the geodesic of 
$(M,\varrho)$ obtained parametrizing $p(\gamma_i)$ by arc length (so that $\ell_i$ is 
the 
actual length of $p(\gamma_i)$). Define the map 
\begin{center}
 \begin{tikzpicture}[>=to]
\matrix(m)[matrix of math nodes,
text height=1.5ex, text depth=0.25ex]
{	H_i\colon	& \hspace{0 pt}[0,\ell_i]\times[1,\infty)&[2 em]   \cone (M)	
	 \\
			&	(s,t)	 	& 	(\widehat\gamma_i(s),t)	\\ 
			};
\path[->,font=\scriptsize]
(m-1-2) edge (m-1-3);
\path[|->,font=\scriptsize]
(m-2-2) edge (m-2-3);
\end{tikzpicture}
\end{center}
Then the pull\=/back of the metric $\varrho_\cone$ is just a standard cone metric 
\[
 H_i^*\varrho_\cone= t^2ds^2+dt^2
\]
and we can see the path $\gamma_i$ as a geodesic path joining the points $(0,t_i)$ and 
$(\ell_i,t_{i+1})$ in this Euclidean cone.

We can now glue this cones together, obtaining a cone of total angle 
$\sum_{i=i}^n\ell_i$:
\[
 \bigcup_{i=1}^n\big([0,\ell_i]\times[1,\infty),t^2ds^2+dt^2\big)
 =\Big(\big[0,{\textstyle\sum_{i=i}^n}\ell_i\big]\times[1,\infty),t^2ds^2+dt^2\Big).
\]
Let $L\coloneqq \sum_{i=i}^n\ell_i$. Then the paths $\gamma_i$ coincide at the gluing 
points, thus they can be joined to form a path between $(0,t_0)$ and $\big(L,t_0\big)$. 

Note that the total angle $L$ is bounded by the diameter of the manifold $\diam(M)$. Indeed, if this was not the case there would be a geodesic $\widehat\beta$ joining $p(x)$ and $p(y)$ in $(M,\varrho)$ having length less then $L$. The subspace of $\CO(M)$ projecting onto $\widehat\beta$ would be isometric to an Euclidean cone of total angle $\ell(\widehat\beta)<L$ and therefore there would be a geodesic in $\CO(M)$ joining $x$ and $y$ of length less than $\ell(\gamma_0)+\cdots+\ell(\gamma_n)\leq \wdist{}(x,y)$, a contradiction.

Notice that also the paths $\widetilde\gamma_i$ can be realised in the cone 
$[0,L]\times 
[0,\infty)$ and they coincide with the projections of the $\gamma_i$'s onto the level 
set 
$[0,L]\times\{t_0\}$, thus we have
\[
 \sum_{i=1}^n l(\widetilde\gamma_i)=Lt_0.
\]

A straightforward computation in Euclidean geometry shows that the distance between 
$(0,t_0)$ and $(L,t_0)$ inside the cone $[0,L]\times [1,\infty)$ satisfies:
\[
 d\big((0,t_0),(L,t_0)\big)\geq \left\{
 \begin{array}{lc}
  2t_0\sin(L/2) & \text{ if }L\leq \pi \\
  2t_0  & \text{ if }L\geq \pi \\
 \end{array}
 \right. 
\]
and since $\sin(x)\geq 2x/\pi$ for $x\leq \pi/2$, we deduce that 
\[
 \sum_{i=1}^n l(\widetilde\gamma_i)=Lt_0
 \leq \max\left\{\frac{\pi}{2},\frac{L}{2}\right\}\, d\big((0,t_0),(L,t_0)\big)
 \leq\max\left\{\frac{\pi}{2},\frac{L}{2}\right\}\sum_{i=1}^nl(\ell_i).
\]

As we already noted that $L$ is bounded by the diameter of $(M,\varrho)$, letting $C\coloneqq \max\big\{\pi/2,\diam(M,d)/2\big\}$ 
yields $\wdist{t_0}\leq C\wdist{}$ on $M\times\{t_0\}$, completing the proof of the lemma.
\end{proof}

We can now collect the results of this section in the following theorem. Here and after, 
an \emph{unbounded sequence of level sets} of a warped cone $\cone_\Gamma (M)$ is a 
sequence of level sets $M\times \{t_n\}\subset \cone_\Gamma (M)$ with $t_n\to\infty$. 
Recall that by Definition \ref{de:expander.metric.spaces} an unbounded sequence of level 
sets forms a family of expanders if and only if it is uniformly quasi\=/isometric to a 
family of expander graphs.

\begin{thm}
 Let $\Gamma=\angles{S}$ act by quasi\=/symmetric homeomorphisms with bounded measure 
distortion on a compact Riemannian manifold 
$(M,\varrho)$. Then one (equivalently, every) unbounded sequence of level sets 
$M\times\{t_n\}$ of the warped cone $\cone_\Gamma (M)$ forms a family of expanders if and 
only if the action is expanding in measure.
\end{thm}

\begin{proof}
 Fix any sequence $t_n$ with $t_n\to\infty$. By Lemma 
\ref{lem:level.sets_q.i_warped.spaces}, the level sets $M\times\{t_n\}$ with the metric 
induced from $\cone_\Gamma (M)$ are all uniformly quasi\=/isometric to the warped metric 
spaces $(M,\wdist{t_n})$. Now, pick a maximal $1/3$\=/separated set 
$Y_{t_n}\subset(M,\wdist{t_n})$ and build the graphs $\CX_{t_n}$ as before. As already 
noted, the spaces $(M,\wdist{t_n})$ are uniformly quasi isometric 
to the graphs $\CX_t$ which in turns are uniformly quasi\=/isometric to the lavish 
approximating graphs $\widetilde G\big(\CV(Y_t)\big)$ by Lemma 
\ref{lem:lavish.approx.graph_q.i_Xt}.
Therefore, the theorem follows from 
Lemma \ref{lem:lavish.graph.expanders_iff_measure.expand}.
\end{proof}

Theorem \ref{thm:intro:warped.cones.exp_iff_meas.exp} can be used to provide restrictions 
to coarse embeddability of warped cones in Hilbert spaces. 

\begin{cor}
\label{cor:warped.cones.dont.coarsely.embed}
 If the action $\Gamma\curvearrowright M$ is expanding in measure, then the warped cone 
$\cone_\Gamma (M)$ does not coarsely embed into any $L^p$ space.
\end{cor}

\begin{rmk}
 Corollary \ref{cor:warped.cones.dont.coarsely.embed} was recently proved in 
\cite{NoSa15}. Their result is valid for a larger class of measures but it only applies 
to actions which are measure preserving. 
\end{rmk}

\section{A spectral criterion for measure preserving actions}
\label{sec:spectral.criterion}

Following \cite{BFGM07} we give the following:

\begin{de}
 Let $G$ be a locally compact group and $E$ a Banach space (either real or complex). We 
say that an action by 
linear isometries $\pi\colon G\curvearrowright E$ (continuous with respect to the weak 
operator topology) has \emph{almost invariant vectors} if there exists a sequence 
$v_n\in E\smallsetminus\{0\}$ such that
\[
 \lim_{n\to\infty} \frac{\diam\big(\pi(K) v_n\big)}{\norm{v_n}}=0
\]
for every compact set $K\subseteq G$.
\end{de}

Let $(X,\nu)$ be a probability space and $\rho\colon\Gamma\curvearrowright X$ a measure 
preserving action. As always, we assume that $\Gamma$ is finitely generated and we fix a 
symmetric generating $S$ with $1\in S$. 
For $1\leq p\leq\infty$ we will denote by $L^p(X)$ the Banach space 
of complex valued functions of $X$ with finite $L^p$ norm (we might have used real 
valued functions instead: all of the following results hold in the real case as well). 
The action on $X$ induces a (continuous) left action by affine isometries
$\pi_\rho\colon\Gamma\curvearrowright L^p(X)$ by pre\=/composition: 
$(g\cdot f)(x)\coloneqq f(g^{-1}\cdot x)$. 

The action $\pi_\rho\colon \Gamma\curvearrowright L^p(X)$ clearly has almost invariant 
vectors because the constant functions are genuine invariant vectors.
The canonical complement of the subspace of constant 
functions in $L^p$ is given by the subspace of functions with zero average
\[
 L^p_0(X)\coloneqq \left\{f\in L^p(X)  \middle | \ \int_X f(x) d\nu(x)=0 \right\}. 
\]
The action of $\Gamma$ preserves $L^p_0(X)$ and we will denote again with $\pi_\rho$ the 
induced action $\pi_\rho\colon \Gamma\curvearrowright L^p_0(X)$.

Notice that the action action $\pi_\rho\colon \Gamma\curvearrowright L^p_0(X)$ has almost 
invariant vectors if and only if 
there exists a sequence $f_n\in L^p_0(X)\smallsetminus \{0\}$ such that
\[
 \lim_{n\to\infty} \frac{\norm{s\cdot f_n -f_n}_p}{\norm{f_n}_p}=0
\]
for every $s\in S$. Equivalently, $\pi_\rho$ does \emph{not} have almost invariant 
vectors 
in $L^p_0(X)$ if and only if there exists a positive constant $\delta>0$ so that 
\[
 \sum_{s\in S}\norm{s\cdot f-f}_p\geq \delta\norm{f}_p
\]
for every function $f\in L^p_0(X)$. In literature, an action with this property is said 
to 
have a \emph{spectral gap} in $L^p_0$.

\begin{lem}
 \label{lem:projecting.doesnt.indrease.Lp.norm}
 For every function $f\in L^p_0(X)$ and every constant $c\in\RR$ we have 
\[
 \norm{\vphantom{\big|}f+c}_p\geq \frac{\norm{f}_p}{2}.
\]
\end{lem}

\begin{proof}
 Let $g$ be any function in $L^p(X)$. Applying Jensen inequality we have
 \[
  \abs{\int_X g(x) d\nu(x)}^p\leq \left(\int_X \abs{g(x)} d\nu(x)\right)^p\leq  \int_X \abs{g(x)}^pd\nu(x). 
 \]
Denote by $\nu(g)$ the average $\int_X g(x) d\nu(x)$. Then we have:
\begin{equation}
 \label{eq:Lp.norm.minus.average}
 \norm{\vphantom{\big|} g(x)-\nu(g)}_p
 \leq \norm{\vphantom{\big|}g}_p+\norm{\nu(g)}_p
 \leq 2\norm{\vphantom{\big|}g}_p.
\end{equation}

Now, for any constant $c$ and any $f\in L^2_0$, the average $\nu(f+c)$ is equal to $c$. 
Thus inequality \eqref{eq:Lp.norm.minus.average} reads as
\[
 \norm{\vphantom{\big|}f}_p\leq 2\norm{\vphantom{\big|}f(x)+c}_p.
\]
\end{proof}

\begin{prop}
\label{prop:almost.inv.independent.p.easy}
Let $\Gamma$ be a finitely generated group and $S$ a finite generating set with 
$S=S^{-1}$ and $1\in S$. Let $\rho\colon\Gamma\curvearrowright (X,\nu)$ 
be a measure preserving action on a probability space. Then, for any $1\leq p<\infty$ the 
induced action $\pi_\rho\colon \Gamma\curvearrowright L^p_0$ does not have almost 
invariant vectors if and only $\rho$ is expanding in measure.
\end{prop}

\begin{proof}

If the action is not expanding in measure, then there exists a sequence of measurable 
sets $A_n$ with 
measure $\nu(A_n)\leq 1/2$ and $\nu(S\cdot A_n)/\nu(A_n)\to 1$. Looking at the symmetric 
difference, we have that 
$\nu\big((s\cdot A_n) \triangle A_n\big)/\nu(A_n)\to 0$ for every $s\in S$. Denote by 
$\infun_{A_n}$ the indicator function of the set $A_n$ and let $f_n(x)\coloneqq 
\infun_{A_n}(x)-\nu(A_n)$. The sequence $\{f_n\}_{n\in\NN}$ lies in $L^p_0(X)$ and we have
\[
 \norm{s\cdot f_n-f_n}_p^p
  = \nu(A_n\smallsetminus s\cdot A_n)+\nu(s\cdot A_n\smallsetminus A_n)
  =\nu\big((s\cdot A_n)\triangle A_n\big)
\]
while
\begin{align*}
  \norm{f_n}^p_p &=\nu(A_n)\big(1-\nu(A_n)\big)^p+\big(1-\nu(A_n)\big)\nu(A_n)^p 
	 \vphantom{\frac{1}{2}} \\    
    & \geq  \nu(A_n)\big(1-\nu(A_n)\big)^p \\
    &\geq \frac{1}{2^p}\nu(A_n) .
\end{align*}
It follows that $(f_n)$ is a sequence of almost invariant vectors in $L^p_0$.

For the converse implication, fix any $1\leq p<\infty$. We 
need to show that if $\rho$ is $\varepsilon$\=/expanding then there is a constant 
$\delta>0$ so that for every function $f\in L^p_0(X)$ we have $\sum_{s\in S}\norm{s\cdot 
f-f}_p\geq\delta\norm{f}_p$.
We can restrict our attention to real functions because $\Gamma$ acts separately on the 
real and imaginary parts. By density, it is then enough to prove 
the statement for scale functions of the form
\[
 f(x)=\sum_{i=0}^N \alpha_i\infun_{A_i}(x)
\]
with $\alpha_i\in \RR$ and $A_N\subseteq A_{N-1}\subseteq \cdots A_0$.

There exists a constant $c$ such that both the set $\big\{x\mid f(x)>c\big\}$ and 
$\big\{x\mid f(x)<c\big\}$ have measure smaller or equal than $1/2$. Let $g\coloneqq 
f-c$, then by Lemma \ref{lem:projecting.doesnt.indrease.Lp.norm} we have that 
$\norm{g}_p\geq\norm{f}_p/2$. Changing sign if 
necessary, we may assume that  
$\norm{g^+}_p\geq\frac{1}{4}\norm{f}_p$ where $g^+=\max\{g,0\}$.

Clearly $\norm{s\cdot f-f}_p=\norm{s\cdot g-g}_p\geq\norm{s\cdot g^+-g^+}_p$, thus 
we only need to find a lower bound for the latter. The function $g^+$ is still a scale 
function
\[
 g^+(x)=\sum_{i=0}^n\beta_i\infun_{B_i}(x)
\]
with $B_{i+1}\subseteq B_i$, but this time we can also assume $\beta_i>0$ and 
$\nu(B_i)\leq1/2$ for every $i=0,\ldots,n$.

Since the $B_i$'s are nested, we have
\begin{align*}
 \abs{s\cdot g^+-g^+}(x)
      &\geq \sum_{i=0}^n \beta_i\abs{ \vphantom{\big(}
		  \infun_{B_i}(s^{-1} x) - \infun_{B_i}(x)} \\
      &\geq \sum_{i=0}^n \beta_i \big(
		  \infun_{B_i\cup sB_i}(x) - \infun_{B_i}(x) \big) \\
      &= h_s(x)-g^+(x)
\end{align*}
where
\[
 h_s(x)\coloneqq \sum_{i=0}^n \beta_i \infun_{B_i\cup sB_i}(x).
\]

Note that
\[
 \sum_{s\in S}\norm{h_s}_p^p
      =\int_{\RR^+}\sum_{s\in S}\nu\big(\{x\mid h_s(x)^p\geq r\}\big)dr.
\]
Since the action is expanding we have
\[
 \sum_{s\in S}\nu\big(\{x\mid h_s(x)^p\geq r\}\big)
    \geq \big(\abs{S}+\varepsilon\big)\nu\big(\{x\mid g^+(x)^p\geq r\}\big);
\]
thus we get:
\begin{align*}
 \sum_{s\in S}\norm{h_s}_p^p
      \geq\int_{\RR^+}\big(\abs{S}+\varepsilon\big)\nu\big(\{x\mid g^+(x)^p\geq 
r\}\big)dr =\vphantom{\int}\big(\abs{S}+\varepsilon\big)\norm{g^+}_p^p
\end{align*}
whence we deduce that there exists $s\in S$ such that $\norm{h_s}_p\geq 
\big(1+\varepsilon/\abs{S}\big)^{1/p}\norm{g^+}_p$.

Let $\delta'\coloneqq \big(1+\varepsilon/\abs{S}\big)^{1/p} -1$, then for the same $s\in 
S$ we have
\[
 \norm{s\cdot g^+-g^+}_p\geq \norm{h_s-g^+}_p\geq \norm{h_s}_p-\norm{g^+}_p\geq 
\delta'\norm{g^+}_p.
\]
and the Lemma follows because
\[
 \norm{s\cdot f-f}_p\geq \norm{s\cdot g^+-g^+}_p\geq \delta'\norm{g^+}_p \geq 
\frac{\delta'}{4}\norm{f}_p.
\]
\end{proof}

Combining Proposition \ref{prop:almost.inv.independent.p.easy} with Theorem 
\ref{thm:action.on.mfld.exp_iff_approx.graph.expand} we obtain the following:

\begin{cor}
\label{cor:warp.cones.exp_iff_spec.gap}
Let $\Gamma=\angles{S}$ act by quasi\=/symmetric homeomorphisms on a compact Riemannian 
manifold $(M,\varrho)$ and assume that the action preserves the Riemannian measure. Then 
one (any) unbounded sequence of level sets of $\cone_\Gamma (M)$ forms a family of 
expanders if and only if the action has a spectral gap.
\end{cor}

\begin{rmk}
 The proof of Proposition \ref{prop:almost.inv.independent.p.easy} provides explicit bounds on the expansion constant $\varepsilon$ in terms of the spectral gap constant $\delta$ and viceversa.
\end{rmk}

\begin{rmk}
 Let $\Gamma_i$ be a sequence of finite index subgroups of $\Gamma=\angles{S}$ with increasing index. Endow the set of left cosets $\Gamma/\Gamma_i$ with the uniform probability measure. Then $\Gamma$ acts on $\Gamma/\Gamma_i$ by multiplication and by considering the complete partition we clearly have that the approximating graphs are actually equal to the Schreier graphs of $\Gamma/\Gamma_i$ with respect to the set $S$. 
 Considering the sequence of actions $\rho_i\colon\Gamma\curvearrowright\Gamma/\Gamma_i$, it follows from Proposition \ref{prop:almost.inv.independent.p.easy} and Lemma \ref{lem:meas.exp___cheeg.const} that the approximating graphs are expanders if and only if all those actions have a uniform spectral gap; \emph{i.e.} we recovered the well\=/known fact that the Schreier graphs form a family of expanders if and only if $\Gamma$ has property $(\tau)$ with respect to the sequence of $\Gamma_i$. 
 
 Alternatively, if the groups $\Gamma_i$ are nested, the Schreier graphs can be obtained as graphs approximating the action on the profinite limit $\varprojlim \Gamma/\Gamma_i$ equipped with the probability measure assigning to a $\Gamma_i$\=/cosets (seen as a subset of $\varprojlim \Gamma/\Gamma_i$) probability $1/[\Gamma:\Gamma_i]$. Expansion can hence by checked by studing the action $\Gamma\curvearrowright \varprojlim \Gamma/\Gamma_i$.
\end{rmk}

\begin{rmk}
 It follows from Proposition \ref{prop:almost.inv.independent.p.easy} that the existence 
of almost invariant vectors for the action $\pi_\rho$ does not depend on what $1\leq 
p<\infty$ is being considered. This was proved in the more general setting of 
measure\=/class preserving actions of topological groups in \cite[Remark 4.3]{BFGM07}. 
They state it for real $L^p$ spaces, but their proof works in the complex case as well.
\end{rmk}

\begin{rmk}
 Proposition \ref{prop:almost.inv.independent.p.easy} also implies that a measure 
preserving action $\Gamma\curvearrowright X$ is expanding if and only if there exists a 
unique invariant mean on $L^\infty(X)$ (see \cite{Ros81} and \cite{Sch81}). Therefore, 
the expansion in measure is 
equivalent to a positive answer to the Ruziewicz problem for this action. 
\end{rmk}

\section{A source of expanding actions: subgroups generated by Kazhdan sets}
\label{sec:finite.Kazhdan.sets}

In this section we will use the spectral criterion from Section 
\ref{sec:spectral.criterion} to link expanding actions to the well developed machinery 
of Kazhdan's property (T). Such link proves to be an invaluable tool in constructing 
explicit examples of expanding actions and it also allows us to reformulate some 
classical theorems and conjectures. 

We begin by recalling some definitions. Here $G$ will always be a locally compact second 
countable Hausdorff 
topological group and we will only consider its continuous unitary representations. Most 
of the tools we use can be found in \cite{Sha00} and \cite{BHV08}.

\begin{de}
\label{de:Kazhdan_pairs}
 Let $G$ be a locally compact topological group. Let $K\subseteq G$ be a compact subset 
and $\varepsilon>0$ a constant. If $\CH$ is a (complex) Hilbert space and 
$\pi\colon G\to U(\CH)$ is a unitary representation, a 
\emph{$(K,\varepsilon)$\=/invariant vector} is a vector $v\in \CH$ such that 
$\norm{\pi(g)v-v}\leq\varepsilon\norm{v}$ for every $g\in K$.

Given a family $\CF$ of unitary representations of $G$, we say that a pair 
$(K,\varepsilon)$ is a \emph{Kazhdan pair} for $\CF$ if every representation $\pi\in\CF$ 
does not have non\=/zero $(K,\varepsilon)$\=/invariant vectors. 
When this is the case, $K$ (resp. $\varepsilon$) 
is a \emph{Kazhdan set} (resp. a \emph{Kazhdan constant}) for $\CF$.
\end{de}

It is easy to verify that a representation $\pi\colon G\to U(\CH)$ admits a Kazhdan pair 
if and only if it does not admit a sequence of almost invariant vectors. 
One can also show that if a family $\CF$ of unitary representations of $G$ admits a 
Kazhdan pair $(K,\varepsilon)$ and $K'$ is any compact generating set for the group $G$, 
then then also $K'$ is a Kazhdan set for $\CF$ (for an appropriate Kazhdan constant 
$\varepsilon'$. See \cite[Chapter 1]{BHV08}).

\begin{de}
A group $G$ is said to have \emph{property (T)} if the family $\CU_0$ of all continuous 
unitary representations without non\=/trivial invariant vectors admits a Kazhdan pair 
$(K,\varepsilon)$. Such pair $(K,\varepsilon)$ (resp. set, constant) is a \emph{Kazhdan 
pair} (resp. \emph{set, constant}) \emph{of 
the group $G$}. In the remainder, when we say that a set $K$ is a Kazhdan set without 
specifying any family of representations we mean that $K$ is a Kazhdan set of the group.
\end{de}

\begin{rmk}
\label{rmk:Kazhdan.set_iff_kazhdan.foreach.rep}
 If a compact subset $K\subseteq G$ is a Kazhdan set for every unitary 
$G$\=/representation $\pi$ with no invariant vector, then it is a Kazhdan set of $G$. 
Indeed, if $\varepsilon_\pi>0$ is the largest constant such that $(K,\varepsilon_\pi)$ is 
a Kazhdan pair for $\pi$, then there must exist a $\varepsilon>0$ such that 
$\varepsilon_\pi\geq\varepsilon$ for every $\pi$; otherwise one would get a contradiction 
by considering the direct sum of a sequence of representations $\pi_n\in\CU_0$ with 
$\varepsilon_{\pi_n}\to 0$.
\end{rmk}

Recall that the \emph{diagonal matrix coefficients} of a unitary 
representation $\pi\colon G\to \U(\CH)$ are the complex functions on $G$ sending
$g\mapsto\angles{\pi(g)v,v}$ where $v$ is any fixed vector in $\CH$. Given two unitary 
representations $\pi,\pi'$ of $G$, $\pi'$ is \emph{weakly contained in $\pi$} (denoted 
$\pi'\prec\pi$) if every diagonal matrix coefficient of $\pi'$ can be approximated 
uniformly on compact sets by convex combinations of matrix coefficients of $\pi$.

Let $I_G$ denote the trivial representation of $G$. Then it is easy to check that a unitary
representation $\pi$ admits a Kazhdan pair if and only if $\pi$ does \emph{not} weakly 
contain the trivial representation ($I_G\nprec \pi$). In particular, a group $G$ has 
property (T) if and only if every unitary representation $\pi$ weakly containing $I_G$ 
has non\=/trivial invariant vectors (\emph{i.e.} it contains $I_G$ as a 
subrepresentation).

Finally, given a probability measure $\mu$ on $G$ and a unitary 
representation $\pi\colon G\to \U(\CH)$, the \emph{$\mu$\=/convolution operator} 
$\pi(\mu)\colon \CH\to\CH$ is defined as the operator such that
\[
 \angles{\pi(\mu)v,u}= \int_G \angles{\pi(g)v,u}d\mu
\]
for every $v,u\in\CH$.

Note that since $\mu$ is a probability measure $\norm{\pi(\mu)}\leq 1$. Moreover, 
if $S\subset G$ is a finite symmetric set with $1\in S$, and $\mu_S\coloneqq 
\frac{1}{\abs{S}}\sum_{s\in S} \delta_s$ is the probability measure equidistributed on 
$S$, then one can show using elementary spectral 
theory that the representation $\pi$ has a spectral gap (in the sense of 
Section \ref{sec:spectral.criterion}) if and only if $\norm{\pi(\mu_S)}<1$ and this is 
equivalent to $r_{sp}\pi(\mu_S)<1$, where $r_{sp}\pi(\mu_S)$ is the \emph{spectral 
radius} of the operator $\pi(\mu_S)$. In particular, $\norm{\pi(\mu_S)}<1$ if and only if 
$S$ is a Kazhdan set for $\pi$.

\

We can now state the motivating result of this section and explain some of its 
applications. Given a finite subset $S\subseteq G$ that is 
symmetric and contains the identity, denote by $\Gamma\coloneqq\angles{S}$ the  
subgroup of $G$ generated by $S$. If $(X,\nu)$ is a probability space and $\rho\colon 
G\curvearrowright X$ is a measure preserving action, we can investigate expansion 
properties of the restriction of $\rho$ to $\Gamma$ and we obtain the following:

\begin{prop} 
\label{prop:meas.pres.action.expand_iff_Kazhdan.couple}
 The restriction $\rho|_\Gamma\colon \Gamma\curvearrowright (X,\nu)$ is expanding in 
measure if and only if there exists 
a constant $\varepsilon>0$ such that $(S,\varepsilon)$ is a Kazhdan pair for the 
representation $\pi_\rho\colon G\curvearrowright L^2_0(X)$.
\end{prop}

\begin{proof}
By Proposition \ref{prop:almost.inv.independent.p.easy} we know that $\rho|_\Gamma$ is 
expanding in measure if and only if $\pi_{\rho|\Gamma}\colon {\Gamma}\curvearrowright 
L^2_0(X)$ does not have almost invariant vectors. That is, if and only if 
$\pi_{\rho|\Gamma}$ admits a Kazhdan pair $(K,\varepsilon)$. 
Since $S$ is a finite generating set of $\Gamma$, 
such a Kazhdan pair exists if and only if there exists a constant $\varepsilon'>0$ 
so that $(S,\varepsilon')$ is itself a Kazhdan pair for $\pi_{\rho|\Gamma}$.

Now, since $\pi_{\rho|\Gamma}=(\pi_\rho)|_\Gamma$ and $S\subseteq \Gamma$, we have have 
that 
$(S,\varepsilon')$ is a Kazhdan pair for $\pi_{\rho|\Gamma}$ if and only if it is a 
Kazhdan pair for $\pi_\rho$ as well.
\end{proof}

\begin{rmk}
When $(S,\varepsilon)$ is a Kazhdan pair for $\pi_\rho$, one can retrieve 
explicit bounds on the expansion constant of $\rho$ in term of the Kazhdan constant 
$\varepsilon$ and vice versa following the proof of 
Proposition \ref{prop:almost.inv.independent.p.easy}.

\end{rmk}

In the rest of this section we describe some consequences of Proposition 
\ref{prop:meas.pres.action.expand_iff_Kazhdan.couple} (we refer the reader to 
\cite[Section 2]{CoGu11} for more examples of actions of finitely generated groups on 
measure spaces that have a spectral gap). 

\

{\bf A characterisation of Kazhdan sets.}
Schmidt, Connes and Weiss characterised groups with Kazhdan's property (T) in term of 
their ergodic actions. Specifically, they proved that $G$ has property (T) 
if and only if for every measure preserving ergodic action on a probability space 
$\rho\colon G\curvearrowright (X,\nu)$ the induced unitary representation $\pi_\rho\colon 
G\curvearrowright L^2_0(X)$ admits a Kazhdan pair. 

Using Remark 
\ref{rmk:Kazhdan.set_iff_kazhdan.foreach.rep}, one can adapt the proof of the 
Schmidt\=/Connes\=/Weiss theorem given in \cite[Theorem 6.3.4]{BHV08} 
to prove the following more precise statement: a compact subset $K\subseteq G$ is a 
Kazhdan set of $G$ if and only if for every ergodic action $\rho$ of $G$ the set $K$ is a 
Kazhdan set for the representation $\pi_\rho$.
Therefore, Proposition 
\ref{prop:meas.pres.action.expand_iff_Kazhdan.couple} implies the following:

\begin{thm}
\label{thm:finite.Kazhdans_iff_every.action.expanding}
Let $S\subset G$ be a finite symmetric set containing the identity and let 
$\Gamma\coloneqq\angles{S}\subset G$.
Then, $S$ is a Kazhdan set of $G$ if and only if the restriction to $\Gamma$ of every 
ergodic action $\rho\colon G\curvearrowright (X,\nu)$ is expanding in measure. 

Moreover, when $S$ is a Kazhdan set of $G$, all the $\Gamma$\=/actions obtained as 
restrictions of ergodic $G$\=/actions share a lower bound on their expansion 
constants depending only on the Kazhdan constant of $S$ in $G$.
\end{thm}

\

{\bf Non\=/compact Lie groups.} In the setting of non\=/compact Lie groups, the work of 
Y.Shalom provides very general means of proving spectral gap properties. In particular, 
he proved the following. 

\begin{thm}[\cite{Sha00}, Theorem C]
\label{thm:shalom.theorem.C}
 Let $G=\prod G_i$ be a semisimple Lie group with finite centre and let $\pi$ be a 
unitary $G$\=/representation and $\mu$ a probability measure on $G$. If either of the 
following is true:
\begin{enumerate}[a$)$]
 \item $I_G\nprec\pi|_{G_i}$ for every $G_i$ and $\mu$ is not supported on a closed 
amenable subgroup of $G$, 
 \item $I_G\nprec \pi$ and for every $i$ the projection of $\mu$ on $G_i$ is not 
supported on a closed amenable subgroup of $G_i$,
\end{enumerate}
then the operator $\mu(\pi)$ has spectral radius strictly less then one: 
$r_{sp}\pi(\mu)<1$.
\end{thm}

As a corollary one can produce a multitude of examples of expanding actions. Indeed, let
$\rho\colon G\curvearrowright(X,\nu)$ be a measure preserving action of a semisimple Lie 
group with finite centre and let $\pi_\rho\colon G\curvearrowright 
L^2_0(X)$ be the induced unitary representation. If $I_G\nprec\pi|_{G_i}$ for every $G_i$ 
(resp. $I_G\nprec \pi$) and $S$ is any finite symmetric set containing the identity 
and so that the closure of 
$\Gamma\coloneqq\angles{S}$ in $G$ is not amenable (resp. the closures of the projections 
of $\Gamma$ to the $G_i$'s are 
not amenable), then the restriction $\rho|_\Gamma\colon \Gamma\curvearrowright (X,\mu)$ 
is expanding in measure by Proposition 
\ref{prop:meas.pres.action.expand_iff_Kazhdan.couple}.

For example, if a simple Lie group $G$ has Kazhdan property (T) and finite 
centre\note{can the centre be infinite?} 
and $\rho\colon G\curvearrowright (X,\nu)$ is any measure preserving ergodic 
action then $I_G\nprec\pi|_{G}$. If $S$ 
generates a discrete subgroup $\Gamma<G$, then $\rho|_\Gamma$ is expanding as soon as 
$\Gamma$ is not amenable (\emph{i.e.} as soon as $\Gamma$ contains a non\=/abelian free 
subgroup). 
A typical example of ergodic action of $G$ is the action by left multiplication 
$G\curvearrowright G/\Lambda$ where $\Lambda<G$ is a lattice. More generally, if $G<G'$ 
where $G'$ is a finite product of connected, non compact, simple Lie groups with finite 
centre and $\Lambda$ is any irreducible lattice of $G'$, then Moore's Ergodicity Theorem 
implies that the action by left multiplication $G\curvearrowright G'/\Lambda$ is ergodic 
if and only if the closure of $G$ in $G'$ is not compact.

\begin{rmk}
Any right invariant Riemannian metric on $G'$ descends 
to a Riemannian metric on $G'/\Lambda$ whose volume form is (a multiple of) the 
restriction of the Haar 
measure and the action on the left $G'\curvearrowright G'/\Lambda$ is by bilipschitz 
diffeomorphisms. In the example above, it follows that if the irreducible lattice 
$\Lambda$ is also uniform, then one 
can consider the warped cone $\cone(G'/\Lambda)$ and by Corollary 
\ref{cor:warp.cones.exp_iff_spec.gap} any unbounded sequence of its level sets forms a 
family of expanders.
\end{rmk}

\begin{rmk}
 In the above example, we assumed $G$ to have property (T), but the subgroup $\Gamma$ 
does not need to have it (nor does $G'$). Indeed, taking $\Gamma$ to 
be any discrete non\=/abelian free group would do. 
This is a very interesting feature, as it allows us to build expanders out of actions of 
free groups and, more generally, of a\=/T\=/menable groups. Moreover, we can produce 
examples of warped cones that do not coarsely embed into 
Hilbert spaces even if the warping group $\Gamma$ has Haagerup property.
\end{rmk}

In the same paper, Shalom constructed explicitly finite Kazhdan sets for algebraic 
groups and he was also able to compute their Kazhdan constants \cite[Theorem A]{Sha00}. 
More precisely, he finds Kazhdan sets of $m$ elements whose Kazhdan constant is 
\[
 \varepsilon=\sqrt{2-2(\sqrt{2m-1}/m)}
\]
(and we already remarked that these estimates immediately translate in estimates for the 
Cheeger constants of the approximating graphs).
As a concrete example, he proves that the matrices
\begin{displaymath}
 \left(\begin{array}{c|c}
  \begin{array}{cc}
  1 & 2 \\
  0 & 1
  \end{array}  & 0 \\
  \hline
  0 & I_{n-2} 
 \end{array}\right),
\left(\begin{array}{c|c}
  \begin{array}{cc}
  1 & 0 \\
  2 & 1
  \end{array}  & 0 \\
  \hline
  0 & I_{n-2} 
 \end{array}\right)
\end{displaymath}
form a Kazhdan set of two elements for $\Sl_n(\RR)$ for every $n\geq 3$.

\

{\bf Compact Lie groups.}
Theorem \ref{thm:shalom.theorem.C} can only be applied to non\=/compact Lie groups. 
Still, the case of compact Lie groups is all but devoid of interest. In fact, one can 
show that every simple, connected, compact Lie group admits finite Kazhdan 
sets (see \cite[Section 5]{Sha99} for more on this). 

Explicit examples are provided by 
Bourgain and Gamburd in \cite{BoGa07}. There they prove that if $k$ elements 
$g_1,\ldots,g_k\in \SU(2)$ generate a free subgroup of $\SU(2)$ and they satisfy a 
non\=/abelian diophantine property then they form a Kazhdan set of $\SU(2)$.
In particular, they show that when two matrices with algebraic entries $a,b\in \SU(2)\cap 
\Gl_2(\overline{\QQ})$ freely generate a 
free group $\Gamma<\SU(2)$, then every ergodic action of $\SU(2)$ restricts to an 
expanding action of $\Gamma$ (here the generating set we should use is actually 
$S=\{1,a^\pm,b^\pm\}$ ). 

An obvious example of an ergodic action of $\SU(2)$ on a 
compact space is the action by left multiplication of $\SU(2)$ on itself. Alternatively, 
note that $\SU(2)$ is the double cover of 
$\SO(3)$ and the action of the latter on the sphere $S^2$ is ergodic. Thus, we obtain the 
following:

\begin{cor}
 Let $a$ and $b$ be two independent rotations of $S^2$ whose matrices have algebraic 
entries and let $F_2=\angles{a,b}$ be the generated subgroup of $\SO(3)$. Then any 
unbounded sequence of level sets of the warped cone $\cone_{F_2}\big(S^2\big)$ forms a 
family of expanders.
\end{cor}

\begin{rmk}
 The existence of actions by rotations on the sphere $S^2$ that are expanding in measure 
has already been successfully used in relation to the Ruziewicz problem and to the 
problem of constructing finite equidistributed subsets of $S^2$ \cite{Lub10}.
\end{rmk}

\

{\bf A conjecture of Gamburd, Jakobson and Sarnak.}
The results of \cite{BoGa07} have been later extended to $\SU(n)$ for any $n\geq 2$ in 
\cite{BoGa10} and subsequently to all compact simple Lie groups  in 
\cite{BeSa14}. These works build on the notion of \emph{non\=/abelian diophantine 
property} introduced 
in \cite{GJS99} in order to study spectral gap properties for generic 
subgroups 
of rotations.

For a generic $k$\=/tuple of elements in $\SU(2)$, it is known that the action on $S^2$ 
of the generated subgroup $\Gamma<\SU(2)$ is ergodic; and it is conjectured in 
\cite{GJS99} that the action of $\Gamma$ should also have a spectral gap 
(which is a much stronger property). A partial result is due to Fisher \cite{Fis06} who 
managed to prove that if the conjecture 
is false then the set of $k$\=/tuples inducing actions with spectral gap must have null 
measure. 
It is unknown whether the group generated by a generic $k$\=/tuple has the 
non\=/abelian diophantine property (an affirmative answer to the latter would clearly 
imply the conjecture).

More generally, it is unknown whether the action by left
multiplication $\Gamma\curvearrowright G$ of a generic finitely generated dense subgroup 
$\Gamma$ of a compact simple Lie group $G$ has a spectral gap. Corollary 
\ref{cor:warp.cones.exp_iff_spec.gap} immediately implies that the last conjecture is 
equivalent to the statement that generic warped cones form families of expanders:

\begin{thm}
 A generic dense subgroup $\Gamma$ of a compact simple Lie group $G$ has a spectral gap 
if and only if one (any) unbounded sequences of the level sets of a generic warped cone 
$\cone_\Gamma(G)$ forms a family of expanders.
\end{thm}

\

{\bf Warped cones and finite Kazhdan sets of compact groups.}
As a last note, we wish to report a nice feature of compact Lie group already noted in 
\cite{Sha99}. Let $G$ be a compact Lie group with Haar measure $\nu$ and consider its 
action on itself by left multiplication. The induced unitary representation $\pi_R\colon 
G\to\U(L^2(G))$ is called the \emph{(left) regular representation} of $G$. 
Notice that the regular representation decompose as the 
direct sum of its restriction to $L^2_0(G)$ and the trivial representation. The 
Peter\=/Weyl theorem implies the following strong version of Theorem 
\ref{thm:finite.Kazhdans_iff_every.action.expanding}:

\note{hidden in BdHV there must be a proof of it}

\begin{lem}
 A finite set $S\subset G$ is a Kazhdan set of $G$ if and only if the restriction of 
$\pi_R |_{L^2_0(G)}$ to the generated group $\Gamma=\angles{S}$ has a spectral gap. 
\end{lem}

\begin{cor}
 Let $G$ be a compact Lie group and let $\Gamma=\angles{S}$ where $S\in G$ is finite 
symmetric subset. Choose any Riemannian metric on $G$. Then the level sets of the warped 
cone $\cone_\Gamma (G)$ form a family of expanders if and only if 
$S$ is a Kazhdan set of $G$.
\end{cor}

\begin{proof}
 The statement is clearly true if the volume form induced by the Riemannian metric 
coincides with the Haar measure $\nu$. Any other Riemannian volume form equals 
$f(x)\nu(x)$ for some strictly positive smooth function $f$. Since $G$ is a compact, 
there 
are constants $0<c<C$ so that $c<f(x)<C$ $\forall x\in G$ and it follows that 
$\Gamma\curvearrowright G$ is expanding in measure with respect to $\nu$ if and only if 
it 
is expanding in measure with respect to $f(x)\nu(x)$.
\end{proof}

\bibliography{MainBibliography}{}
\bibliographystyle{amsalpha}

\end{document}